\documentclass[leqno,12pt,a4paper]{amsart}
\usepackage{amsmath,amssymb,amsthm,verbatim,color,tikz-cd}
\usepackage[T1]{fontenc}
\usepackage{geometry}

\theoremstyle{plain}
\newtheorem{theorem}{Theorem}
\newtheorem{lemma}[theorem]{Lemma}
\newtheorem{corollary}[theorem]{Corollary}
\newtheorem{proposition}[theorem]{Proposition}

\theoremstyle{definition}
\newtheorem{definition}{Definition}
\newtheorem{question}{Question}

\newtheorem{remark}[theorem]{Remark}

\theoremstyle{remark}
\newtheorem*{claim}{Claim}

\newcommand{\N}{\mathbb N}
\newcommand{\Z}{\mathbb Z}

\newcommand{\R}{\mathbb R}
\newcommand{\C}{\mathbb C}
\newcommand{\Bai}{\N^{\N}}

\newcommand{\card}{{\rm{card}}}

\newcommand{\pol}{\le_p}
\newcommand{\npol}{\nleq_p}
\newcommand{\lpol}{<_p}
\newcommand{\eqpol}{\equiv_p}
\newcommand{\chpl}{\mathcal F}

\newcommand{\quot}[2]{{\raisebox{.2em}{$#1\!$}\left/\raisebox{-.2em}{$#2$}\right.}}

\title{Reducibility by polynomial functions}

\author{Riccardo Camerlo}
\address{Dipartimento di matematica, Universit\`a di Genova, Via Dodecaneso 35, 16146 Genova --- Italy}
\email{camerlo@dima.unige.it}
\author{Carla Massaza}
\address{Via Muriaglio 19, 10141 Torino --- Italy}
\email{carla.massaza@formerfaculty.polito.it}


\date{}
\begin{document}

\maketitle

\begin{abstract}
We study the preorder $ \pol $ on the family of subsets of an algebraically closed field of characteristic $0$ defined by letting $A \pol B $ if there exists a polynomial $P$ such that $A=P^{-1}(B)$.
\end{abstract}

\tableofcontents

\section{Introduction and plan of the paper}
We let $k$ always denote a commutative field, and $\kappa $ its cardinality; in a moment we shall add the assumption that $k$ is algebraically closed and has characteristic $0$.
Purpose of the paper is to study the following relation between subsets of $k$.

\begin{definition}
For $A,B\subseteq k$, let $A \pol B$ if there is a polynomial function $P:k\to k$ such that
\begin{equation} \label{eqndfnt}
A=P^{-1}(B).
\end{equation}
We then say that $A$ \emph{polynomially reduces} to $B$.
A polynomial $P$ as in \eqref{eqndfnt} is said to \emph{reduce} $A$ to $B$.
Also, let $A \eqpol B$ denote $A \pol B \pol A$, in which case we say that $A,B$ are \emph{polinomially bireducible}.
If $A \pol B$ but $A \not\equiv_pB$, we write $A<_pB$.
\end{definition}

Relation $ \pol $ is an instance of a \emph{reducibility relation}: if $ \mathcal F $ is a class of functions from a set $X$ to itself, containing the identity function and closed under composition, setting
\begin{equation} \label{25102022}
A\le_{ \mathcal F }B\Leftrightarrow\exists f\in \mathcal F,A=f^{-1}(B),
\end{equation}
for $A,B\subseteq X$, defines a preorder on $ \mathcal P (X)$, the powerset of $X$.
This preorder is called $ \mathcal F $-\emph{reducibility} and the function $f$ appearing in \eqref{25102022} is called a \emph{reduction} of $A$ to $B$.

The idea underlying the notion of reducibility is that, whenever the relation $A\le_{ \mathcal F }B$ is witnessed by a reduction $f\in \mathcal F $, the complexity of the set $A$ is bounded by the complexity of the set $B$ since the problem of recognising membership in $A$ is reduced to the same problem for $B$, as $x\in A\Leftrightarrow f(x)\in B$.
This notion of relative complexity depends on the chosen class of reductions $ \mathcal F $, and the most notable example is when $X$ is a topological space and $ \mathcal F $ is the collection of continuous functions on $X$.
Then $\le_{ \mathcal F }$ is called \emph{continuous reducibility}, or \emph{Wadge reducibility}, since the first systematic study of it (for $X=\Bai $) was performed in Wadge's thesis \cite{wadge1983}.

In this paper we are interested in the case our set is an algebraically closed field $k$ of characteristic $0$, and we choose as $ \mathcal F $ the class of polynomial functions on $k$, since these are the basic functions of interest in commutative algebra and algebraic geometry, being the morphisms of the affine variety $k$.
As $ \eqpol $ is an equivalence relation, we still denote $ \pol $ the partial order induced on the equivalence classes, that is on $ \quot{ \mathcal P (k)}{ \eqpol } $.
The $ \eqpol $-equivalence classes $[A]=\{ B\in \mathcal P (k)\mid B \eqpol A\} $ will be called \emph{polynomial classes}.

We focus mainly on the study of $ \pol $ on finite subsets of $k$: this is the content of section \ref{17092022}, which constitutes the core of the paper.
After presenting some general facts, in section \ref{170920221559} we prove that polinomial classes of sets of $n$ elements, for any given $n$, form an affine variety $\Sigma_n$ of dimension $n-2$.
The import of this fact is that it provides a framework to study $ \pol $ quantitatively, by using dimension, which is done in the subsequent sections.
In section \ref{170920221620} we describe $ \eqpol $ on finite sets from a geometrical viewpoint and in section \ref{170920221621} we prove that $\Sigma_3$ is an affine line.
Section \ref{excclasses} is devoted to those polynomial classes that have less representatives than the generic ones.
In this sense, we call them \emph{exceptional classes} and we prove in proposition \ref{170920221714} and corollary \ref{170920221716} that they form a subvariety of $\Sigma_n$ of dimension $ \frac{n-2}2 $ or $ \frac{n-3}2 $, according to whether $n$ is even or odd.
In section \ref{18092022} we prove several results about $ \pol $ on finite sets.
Given a $n$-element subset $A$ of $k$, we give conditions for a $m$-element set to be comparable with $A$ and we prove that the family of polynomial classes of $m$-element sets comparable with $[A]$ is a subvariety of $\Sigma_m$ of which we compute the dimension (proposition \ref{180920220929} and theorems \ref{bidimvar} and \ref{030820221312}).
In theorem \ref{propactwo} we show that most polynomial classes of finite sets are $ \pol $-maximal below the class of $1$-element sets, thus providing a big antichain for $ \pol $.

In section \ref{180920220956} we gather some results about $ \pol $ on sets that are both infinite and coinfinite.
We show that there are $2^{\kappa }$ polynomial classes of sets of cardinality $\kappa $ and whose complement has cardinality $\kappa $ that are $ \pol $-maximal (theorem \ref{maximalinfinite}).
Moreover, if $k$ is countable, most subsets of $k$ are $ \pol $-maximal from the Baire category point of view, in the sense that they form a comeagre subset of the powerset of $k$ (theorem \ref{050420221932}).
Finally, when $k\subseteq \C $, we construct a chain of polynomial classes of order type $\zeta $ ---the order type of the integers--- in proposition \ref{longzchain}.

To establish one of our main results, we need in lemma \ref{020820221452} that a specific matrix, which is an enrichment of the Vandermonde matrix, is non-singular.
This observation, which we could not find in the existing literature, might be interesting on its own and its proof is given in a separate section (section \ref{310720221554}).

\begin{remark} \label{firstremarks}
\begin{enumerate}
\item Denote by $Pol$ the monoid of polynomial functions on $k$ under the operation of composition.
The preorder $ \pol $ is induced by the right action of $Pol$ on $ \mathcal P (k)$ defined by $(A,P)\mapsto P^{-1}(A)$.
\item There are two incomparable polynomial classes that precede every other polynomial class: $[\emptyset ]=\{\emptyset\} $ and $[k]=\{ k\} $.
Indeed, if $A$ is a proper subset of $k$, then any constant polynomial with value outside $A$ is a reduction of $\emptyset $ to $A$; similarly, $k$ is polynomially reducible to any non-empty set $A$.
\end{enumerate}
\end{remark}

Polynomial functions are continuous with respect to the Zariski topology on $k$ (which is the cofinite topology), so it follows that
\begin{equation} \label{wadgevspol}
\forall A,B\in \mathcal P (k),A\le_pB\Rightarrow A\le_WB
\end{equation}
where $\le_W$ is the relation of Wadge reducibility with respect to the Zariski topology.
The relation  $\le_W$ for Zariski topologies has been studied in \cite{cammas}: for example, it is noted there that non-empty, proper, closed subsets of $k$ (when $k$ is infinite these are exactly the finite non-empty sets) form a single class with respect to the equivalence relation associated to $\le_W$.
Proposition \ref{orderedfields} below shows that for some fields $k$ the relation $ \pol $ can coincide with $\le_W$ on finite sets.
However, purpose of this note is to illustrate that in general the situation is quite different.

\begin{proposition} \label{orderedfields}
Let $A$ be a finite, non-empty subset of $k$.
Then $A \pol \{ 0\} $.

If $k$ is an ordered field, then $\{ 0\} \pol A$ holds as well.
Consequently, if $k$ is an ordered field then $A \eqpol B$ for any finite, non-empty subsets $A,B$.
\end{proposition}

\begin{proof}
The polynomial $P(X)=\prod_{a\in A}(X-a)$ reduces $A$ to $ \{ 0\} $.
If $k$ is an ordered field, let $a=\max A$; then $Q(X)=X^2+a$ reduces $\{ 0\} $ to $A$.
\end{proof}

\fbox{From now on $k$ is algebraically closed and has characteristic $0$.}

\begin{remark} \label{04042022732}
The assumption on $k$ implies in particular that:
\begin{enumerate}
\item $k$ is infinite, so this allows to identify polynomials and polynomial functions.
\item Every non-constant polynomial function is surjective; therefore if the non-constant polynomial $P$ reduces $A$ to $B$, it also follows that $P(A)=B$.
\end{enumerate}
\end{remark}

\begin{proposition} \label{cardinality}
Let $A\in \mathcal P (k)\setminus\{\emptyset ,k\} $.
If $A \pol B$, then either $A,B$ are both finite, in which case $ \card (A)\ge \card (B)$, or $ \card (A)= \card (B)$.
Consequently, for any $A,B$, if $A \eqpol B$ then $ \card (A)= \card (B)$.
\end{proposition}

\begin{proof}
This holds since non-constant polynomials are surjective and every element has finitely many preimages.
\end{proof}

For $\lambda\le\kappa $ let
\[
\mathcal P_{\lambda }=\{ A\in \mathcal P (k)\mid \card (A)=\lambda\} ,\qquad \check{ \mathcal P }_{\lambda } =\{ k\setminus A\}_{A\in \mathcal P_{\lambda }}.
\]
By proposition \ref{cardinality}, every $ \mathcal P_{\lambda }$ and $ \check{ \mathcal P }_{\lambda }$, as well as $ \mathcal P_{\kappa }\cap \check{ \mathcal P }_{\kappa }$, is invariant with respect to $\eqpol $.
Recall that $ \card ( \mathcal P_{\lambda })= \card ( \check{ \mathcal P }_{\lambda })=\kappa^{\lambda }$ (see \cite[exercise I.13.20]{kunen2012}) and that $ \card ( \mathcal P_{\kappa }\cap \check{ \mathcal P }_{\kappa })=2^{\kappa }$.

We denote $\Sigma_{\lambda }=\{ [A]\mid A\in \mathcal P_{\lambda }\} $: this is the quotient of $ \mathcal P_{\lambda }$ with respect to the restriction of $ \eqpol $.
Similarly, define $ \check{\Sigma }_{\lambda }=\{ [A]\mid A\in \check{ \mathcal P }_{\lambda }\} $.
The results on polynomial reducibility for $ \mathcal P_{\lambda },\Sigma_{\lambda }$ transfer to $ \check{ \mathcal P }_{\lambda }, \check{\Sigma }_{\lambda }$, respectively, via the bijection $A\mapsto k\setminus A$ and the fact that $A \pol B\Leftrightarrow k\setminus A \pol k\setminus B$.

\subsection{Some terminology and notation}
The terminology and notation that we employ throughout the paper is quite standard: main references are \cite{hartsh1977,Harris1992,perrin2008,coxlit2015}.
We collect here a list of some common notions we use.

\begin{itemize}
\item We denote $\deg (P)$ the degree of a polynomial $P$.
A \emph{linear polynomial} is a polynomial of degree $1$, that is, of the form $cX+c'$ where $c\ne 0$.
We denote $ \mathcal L $ the set of linear polynomials; this is a group under composition.
A \emph{quadratic polynomial} is a polynomial of degree $2$, that is, of the form $cX^2+c'X+c''$ where $c\ne 0$.
If $P(X)=cX+c'$ is a linear polynomial, then repeated applications of $P$ give
\begin{equation} \label{2012202218}
P^j(X)=c^jX+(c^j+\dots +c+1)c';
\end{equation}
consequently, if moreover $c\ne 1$, then
\begin{equation} \label{201220221809}
P^j(X)=c^jX+ \frac{c^j-1}{c-1} c'.
\end{equation}
\item A polynomial $h\in k[X_1,\ldots ,X_n]$ is \emph{translation invariant} if
\[
\forall a_1,\ldots ,a_n,t\in k,h(a_1,\ldots ,a_n)=h(a_1+t,\ldots ,a_n+t).
\]
\item If $ \mathcal S \subseteq k[X_1,\ldots ,X_n]$, we denote $V( \mathcal S )=\{ (a_1,\ldots ,a_n)\in k^n\mid\forall f\in \mathcal S ,f(a_1,\ldots ,a_n)=0\} $; when $f_1,\ldots ,f_m\in k[X_1,\ldots ,X_n]$ we also let $V(f_1,\ldots ,f_m)=V(\{ f_1,\ldots ,f_m\} )$.
Sets isomorphic to some $V( \mathcal S )$ are \emph{affine varieties}.

Note that by our definition an affine variety may be empty.
\item If $f\in k[X_1,\ldots ,X_n]$, the set $D(f)=k^n\setminus V(f)$ is a \emph{standard open set}: it consists of all elements of $k^n$ where $f$ does not vanish.
By \cite[proposition III.3.3]{perrin2008}, $D(f)$ is an affine variety endowed with the restriction of the sheaf of regular functions of $k^n$.
\item If $V\subseteq k^n$ is an affine variety, we denote $I(V)=\{ f\in k[X_1,\ldots ,X_n]\mid\forall (a_1,\ldots ,a_n)\in V,f(a_1,\ldots ,a_n)=0\} $ the ideal of $V$.
\item We denote $Sym(A)$ the symmetric group on a set $A$; we let also $Sym_n=Sym (\{ 1,\ldots ,n\} )$.
\item If $\alpha $ is a group action and $a$ is an element of the set acted upon, we denote $Stab_a^{\alpha }$ the stabiliser of $a$ under action $\alpha $.
\end{itemize}

\section{Reducibility on finite sets} \label{17092022}
From \eqref{wadgevspol} it follows that if $B$ is finite, so closed in the Zariski topology, and $A \pol B$, then either $A=k$ or $A$ is finite.
Thus finite sets, together with $k$, constitute an initial segment of $ \pol $.

\begin{proposition} \label{incrcard}
\begin{enumerate}
\item If $1\le n<m,A\in \mathcal P_n,B\in \mathcal P_m$, then $A \npol B$.
\item For every $n\ge 1$, every $A\in \mathcal P_n$, and every $B\in \mathcal P_1$,  it holds that $A \pol B$.
\item $\forall A,B\in \mathcal P_2,A \eqpol B$.
\end{enumerate}

In particular, $\Sigma_1$ and $\Sigma_2$ both consist of a single element.
\end{proposition}

\begin{proof}
(1) This is proposition \ref{cardinality}.

(2) Let $B=\{ b\} $.
Then the polynomial $\prod_{a\in A}(X-a)+b$ reduces $A$ to $B$.

(3) Let $A=\{ a_1,a_2\} ,B=\{ b_1,b_2\} $, in order to show the existence of a polynomial $P:k\to k$, in fact a linear one, such that $A=P^{-1}(B)$.
Let $P(X)=cX+c'$.
The coefficients of such a polynomial can be found as the solutions of the linear system
\begin{equation} \label{eqnincrcard} \left ( \begin{matrix}
a_1 & 1 \\
a_2 & 1 \\
\end{matrix} \right ) \left ( \begin{matrix}
c \\
c'
\end{matrix} \right )
= \left ( \begin{matrix}
b_1 \\
b_2
\end{matrix} \right )
.
\end{equation}
This shows that $A\subseteq P^{-1}(B)$; the fact that $P$ cannot be constant implies that $P^{-1}(B)$ has cardinality $2$ and yields $A=P^{-1}(B)$.

As the polynomial $P^{-1}$ reduces $B$ to $A$, the assertion is proved.
\end{proof}

\begin{proposition} \label{linearreduction}
Let $n\ge 2$.
If $A,B\in \mathcal P_n$ and $P$ is a polynomial such that $A=P^{-1}(B)$, then $P$ is linear.
\end{proposition}

\begin{proof}
Let $A=\{ a_1,\ldots ,a_n\} ,B=\{ b_1,\ldots ,b_n\} $, where it can be assumed that $P(a_j)=b_j$ for every $j$.
Suppose, toward contradiction, that $P$ has degree $\gamma >1$.
Then each equation $P(x)-b_j=0$ has $a_j$ as its unique solution, with multiplicity $\gamma $, so that $P(X)-b_j=c(X-a_j)^{\gamma }$ for some $c\ne 0$.
This means that for every $j,j'$, the polynomial $c(X-a_j)^{\gamma }-c(X-a_{j'})^{\gamma }$ is constant, which is not the case for $j\ne j'$ as the coefficient of degree $\gamma -1$ is $c\gamma a_{j'}-c\gamma a_j\ne 0$.
\end{proof}

\begin{corollary} \label{reductioniffequivalence}
If $A,B\in \mathcal P_n$ are such that $A \pol B$, then $A \eqpol B$.
In other words, the elements of $\Sigma_n$ are pairwise $ \pol $-incomparable.
\end{corollary}

\begin{proof}
By proposition \ref{incrcard}(2), it can be assumed that $n\ge 2$.
If $A \pol B$, by proposition \ref{linearreduction} there is a linear polynomial $P$ such that $A=P^{-1}(B)$.
Then $P^{-1}$ is a polynomial, in fact a linear one, reducing $B$ to $A$.
\end{proof}

Therefore, if $A\in \mathcal P_n$, the polynomial class $[A]$ is the family of all images of $A$ under some linear polynomial.
This leads to the following.

\begin{definition} \label{20112022}
Let $\alpha $ be the action of $ \mathcal L $ on $ \mathcal P (k)$ defined by $\alpha (P,A)=P(A)$.
\end{definition}

We also denote $\alpha $ the restriction of this action to a fixed $ \mathcal P_n$.

\begin{corollary} \label{coractionalpha}
The restriction $ \eqpol^n$ of $ \eqpol $ to $ \mathcal P_n$ is the orbit relation of the action $\alpha $.
\end{corollary}

Corollary \ref{coractionalpha} do not extend to infinite sets: see remark \ref{060520221239}.

\begin{remark} \label{fewpols}
From proposition \ref{linearreduction} and the proof of proposition \ref{incrcard}(3) it follows that for every $A,B\in \mathcal P_2$ there are exactly two polynomials reducing $A$ to $B$: one has as coefficients the solution $(c,c')$ of \eqref{eqnincrcard}; the other one is obtained by switching $b_1,b_2$ in \eqref{eqnincrcard}.
A similar argument applied to the equation
\begin{equation} \label{040820221015}
\left ( \begin{matrix}
a_1 & 1 \\
a_2 & 1 \\
\ldots & \ldots \\
a_n & 1
\end{matrix} \right ) \left ( \begin{matrix}
c \\
c'
\end{matrix} \right )
= \left ( \begin{matrix}
b_1 \\
b_2 \\
\ldots \\
b_n
\end{matrix} \right ) \end{equation}
shows that if $\{ a_1,\ldots ,a_n\} \pol \{ b_1,\ldots ,b_n\} $, then there are at most $n!$ polynomials reducing $\{ a_1,\ldots ,a_n\} $ to $\{ b_1,\ldots ,b_n\} $.
Note indeed that \eqref{040820221015} may be compatible for some enumerations of $\{ b_1,\ldots ,b_n\} $ and incompatible for others.

In particular, the stabiliser $Stab_A^{\alpha }$ of every $A\in \mathcal P_n$ has at most $n!$ elements.
The cardinality of $Stab_A^{\alpha }$ plays a role in the geometry of polynomial classes: see remark \ref{sensstab}.
\end{remark}

\subsection{The variety $\Sigma_n$} \label{170920221559}
Corollary \ref{coractionalpha} suggests the possibility of endowing the orbit space of action $\alpha $, that is $\Sigma_n$, with a geometric structure.
In fact, in this section we show that $ \mathcal P_n$ and, more importantly, $\Sigma_n$ can be given the structure of a quotient affine variety.
The import of this fact is that it provides a quantitative framework to study the relation $ \pol $.

Let
\[
f_n(X_1,\ldots ,X_n)=\prod_{j\ne j'}(X_j-X_{j'})\in k[X_1,\ldots ,X_n].
\]
Then the standard open set $D(f_n)=k^n\setminus V(f_n)$ consists of all elements of $k^n$ having distinct coordinates.
In particular, the ring of regular functions on $D(f_n)$ is the localised ring $\Gamma (D(f_n))=k[X_1,\ldots ,X_n]_{f_n}= \left \{
\frac g{f_n^r} \mid g\in k[X_1,\ldots ,X_n],r\in \N
\right \} $.

\begin{remark}
The polynomial $f_n$ is not the least degree polynomial that identifies $D(f_n)$: one could use the Vandermonde determinant $\prod_{j<j'}(X_j-X_{j'})$.
The polynomial $f_n$ has the advantage of being symmetric, which is convenient in the next steps.
Note also that $f_n$ is homogeneous of degree $n(n-1)$ and translation invariant.
\end{remark}

Next we describe the structure of $ \mathcal P_n$ as an affine variety.
We identify $ \mathcal P_n$ as the orbit space of the action of $Sym_n$ on $D(f_n)$ by permutation of the coordinates.
Then $ \mathcal P_n$ is an affine variety by \cite[\S 10]{Harris1992}: it is the geometric quotient of $D(f_n)$ by this action.
Denote $\theta :D(f_n)\to \mathcal P_n$ the quotient projection.
The ring $\Gamma ( \mathcal P_n)$ of regular functions on $ \mathcal P_n$ consists of those elements of $\Gamma (D(f_n))$ that are invariant with respect to the action of $Sym_n$; this means that if $F\in\Gamma ( \mathcal P_n)$ then there exist a symmetric $g\in k[X_1,\ldots ,X_n]$ and $r\in \N $ such that for every $A=\{ a_1,\ldots ,a_n\}\in \mathcal P_n$ one has
\begin{equation} \label{justf}
F(A)= \frac{g(a_1,\ldots ,a_n)}{(f_n(a_1,\ldots ,a_n))^r} .
\end{equation}

By corollary \ref{coractionalpha}, $\Sigma_n$ is the orbit space of the action $\alpha $ on $ \mathcal P_n$.
Let $\theta': \mathcal P_n\to\Sigma_n$ be the quotient projection.
We claim that $\Sigma_n$ is the geometric quotient of $ \mathcal P_n$ by $\alpha $.
For this, we first identify which is the $k$-algebra $ \mathcal A $ consisting of the elements of $\Gamma ( \mathcal P_n)$ that are invariant under $\alpha $: if $\Sigma_n$ is a geometric quotient, this is to be the ring of regular functions $\Gamma (\Sigma_n)$.

\begin{lemma} \label{lemnnrgamma}
Let $F,g,r$ be as in \eqref{justf}.
Then $F\in \mathcal A $ if and only if $g$ is homogeneous of degree $n(n-1)r$ and translation invariant.
\end{lemma}

\begin{proof}
$(\Rightarrow )$
Assume $F\in \mathcal A $.

Let $\{ a_1,\ldots ,a_n\}\in \mathcal P_n,s\in k\setminus\{ 0\} $.
Since $\{ a_1,\ldots ,a_n\} ,\{ sa_1,\ldots ,sa_n\} $ belong to the same orbit, as witnessed by the linear polynomial $P(X)=sX$, from $F(\{ a_1,\ldots ,a_n\} )=F(\{ sa_1,\ldots ,sa_n\} )$ it follows that
\[
\frac{g(a_1,\ldots ,a_n)}{(f_n(a_1,\ldots ,a_n))^r} = \frac{g(sa_1,\ldots ,sa_n)}{(f_n(sa_1,\ldots ,sa_n))^r} = \frac{g(sa_1,\ldots ,sa_n)}{s^{n(n-1)r}(f_n(a_1,\ldots ,a_n))^r}
\]
whence $g(sa_1,\ldots ,sa_n)=s^{n(n-1)r}g(a_1,\ldots ,a_n)$, showing that $g$ is homogeneous of degree $n(n-1)r$.

Let now $\{ a_1,\ldots ,a_n\}\in \mathcal P_n,t\in k$.
Since $\{ a_1,\ldots ,a_n\} ,\{ a_1+t,\ldots ,a_n+t\} $ belong to the same orbit, as witnessed by the linear polynomial $P(X)=X+t$, from $F(\{ a_1,\ldots ,a_n\} )=F(\{ a_1+t,\ldots ,a_n+t\} )$ it follows that
\[
\frac{g(a_1,\ldots ,a_n)}{(f_n(a_1,\ldots ,a_n))^r} = \frac{g(a_1+t,\ldots ,a_n+t)}{(f_n(a_1+t,\ldots ,a_n+t))^r} = \frac{g(a_1+t,\ldots ,a_n+t)}{(f_n(a_1,\ldots ,a_n))^r}
\]
whence $g(a_1+t,\ldots ,a_n+t)=g(a_1,\ldots ,a_n)$, showing that $g$ is translation invariant.

$(\Leftarrow )$
Assume $g$ is homogeneous of degree $n(n-1)r$ and translation invariant.
Let $A,B\in \mathcal P_n$ be in the same orbit.
If $A=\{ a_1,\ldots ,a_n\} $, then $B=\{ sa_1+t,\ldots ,sa_n+t\} $ for some $s,t$, with $s\ne 0$.
Then
\begin{multline*}
F(B)= \frac{g(sa_1+t,\ldots ,sa_n+t)}{(f_n(sa_1+t,\ldots ,sa_n+t))^r} = \frac{g(sa_1,\ldots ,sa_n)}{(f_n(sa_1,\ldots ,sa_n))^r} = \\
= \frac{s^{n(n-1)r}g(a_1,\ldots ,a_n)}{s^{n(n-1)r}(f_n(a_1,\ldots ,a_n))^r} = \frac{g(a_1,\ldots ,a_n)}{(f_n(a_1,\ldots ,a_n))^r} =F(A).
\end{multline*}
Therefore $F\in \mathcal A $.
\end{proof}

It remains to show that the $k$-algebra $ \mathcal A $, which is a ring of functions $\Sigma_n\to k$, generates the quotient topology on $\Sigma_n$, that is the following.
\begin{proposition}
\begin{enumerate}
\item For every $F\in \mathcal A $, the zero set of $F$ is closed in $\Sigma_n$.
\item For every closed $C\subseteq\Sigma_n$ there exists $ \mathcal B \subseteq \mathcal A $ such that $C$ is the zero set of $ \mathcal B $.
\end{enumerate}
\end{proposition}

\begin{proof}
(1) Let $g,r$ as in \eqref{justf}.
Then the zero set of $F$ in $\Sigma_n$ is the image under $\theta'$ of the zero set determined by $g$ in $ \mathcal P_n$; the latter is a closed subset of $ \mathcal P_n$ invariant under $\alpha $, so its projection under $\theta'$ is closed in $\Sigma_n$.

(2) If $C$ is closed in $\Sigma_n$, there exists a closed set $C'\subseteq \mathcal P_n$ invariant under action $\alpha $ such that $C=\theta' (C')$.
In turn there exists a closed set $C''\subseteq D(f_n)$ invariant under permutation of the coordinates and such that $C'=\theta (C'')$.
Note that $C''$ is also invariant under the tranformations $(a_1,\ldots ,a_n)\mapsto (sa_1,\ldots ,sa_n)$, for $s\ne 0$, and $(a_1,\ldots ,a_n)\mapsto (a_1+t,\ldots ,a_n+t)$.

\begin{claim}
There exists $ \mathcal S \subseteq k[X_1,\ldots ,X_n]$ such that $C''=V( \mathcal S )$ and every element of $ \mathcal S $ is homogeneous, invariant under translation, and symmetric.
\end{claim}

\begin{proof}[Proof of the claim]
Since $C''$ is symmetric, there exists a set $ \mathcal T_0\subseteq k[X_1,\ldots , X_n]$ of symmetric polynomials such that $C''$ is the zero set of $ \mathcal T_0$ in $D(f_n)$.
To see this, notice that $I(C'')$ is a symmetric ideal, that is given any $f\in I(C'')$ any polynomial obtained from $f$ by a permutation of the variables is in $I(C'')$; by \cite[proposition 2.6.4]{sturmf2008} (which is stated for $ \C $, though the proof holds for any algebraically closed field), the radical of the ideal generated by the symmetric polynomials of $I(C'')$ equals $I(C'')$.
Therefore it is enough to take as $ \mathcal T_0$ the set of all symmetric polynomials in $I(C'')$.

The next step is to replace every $h\in \mathcal T_0$ with a set of symmetric and homogeneous polynomials.
The argument is similar to the proof that a projective variety admits homogeneous equations (see for instance \cite[proposition-definition 4.1]{perrin2008}).
Let $h=\sum_{j=1}^uh_j$, where the polynomials $h_j$ are homogeneous and have pairwise different degree ${\gamma }_j$; notice that each $h_j$ is symmetric too.
If $h_1(a_1,\ldots ,a_n)=\ldots =h_u(a_1,\ldots ,a_n)=0$, then $h(a_1,\ldots ,a_n)=0$.
Conversely, if $h(a_1,\ldots ,a_n)=0$ for all $h\in \mathcal T_0$, then the properties of $C''$ imply that, for every $s\ne 0$,
\[
0=h(sa_1,\ldots ,sa_n)=\sum_{j=1}^uh_j(a_1,\ldots ,a_n)s^{\gamma_j}
\]
so that $h_1(a_1,\ldots ,a_n)=\ldots =h_u(a_1,\ldots ,a_n)=0$.
Therefore $V( \mathcal T_0)=V( \mathcal T_1)$,  where $ \mathcal T_1$ consists of the homogenous parts of every degree of all the elements of $ \mathcal T_0$.

The last step is to replace every element of $ \mathcal T_1$ with an equivalent symmetric, homogeneous, and translation invariant polynomial.
For every $h\in \mathcal T_1$, let $p_h(X_1,\ldots ,X_n)=h(X_1- \frac 1n \sum_{j=1}^nX_j,\ldots ,X_n- \frac 1n \sum_{j=1}^nX_j)$.
A direct inspection using the properties of $h$ shows that $p_h$ is indeed symmetric, homogeneous, and translation invariant.
Moreover, by the properties of $C''$,
\begin{multline*}
\forall h\in \mathcal T_1,p_h(a_1,\ldots ,a_n)=0\qquad \text{iff} \\
\forall h\in \mathcal T_1,h \left ( a_1- \frac 1n \sum_{j=1}^na_j,\ldots ,a_n- \frac 1n \sum_{j=1}^na_j \right ) =0\qquad \text{iff} \\
\forall h\in \mathcal T_1,h(a_1,\ldots ,a_n)=0
\end{multline*}
so letting $ \mathcal S =\{ p_h\}_{h\in \mathcal T_1}$ establishes the claim.
\end{proof}

Now, for every $h\in \mathcal S $, let $\gamma_h$ be the degree of $h$.
It follows that the polynomial $(h(X_1,\ldots ,X_n))^{n(n-1)}$ is homogeneous of degree $n(n-1)\gamma_h$, translation invariant, and symmetric.
Setting $ \mathcal B =\left \{
\frac {h^{n(n-1)}}{f_n^{\gamma_h}}
\right \}_{h\in \mathcal S }$
concludes the proof.
\end{proof}

A first immediate property of the variety $\Sigma_n$ is the following.

\begin{corollary}
The affine variety $\Sigma_n$ is irreducible.
\end{corollary}

\begin{proof}
Indeed, $\Sigma_n$ is the image under the morphism $\theta'\theta $ of the irreducible variety $D(f_n)$.
Then apply \cite[proposition I.6.11(2)]{perrin2008}.
\end{proof}

Moreover $\dim (\Sigma_n)=n-2$, see corollary \ref{02102022}(2).

\begin{question} \label{conjnnmt}
Identify $\Sigma_n$.
\end{question}

We answer question \ref{conjnnmt} for $n=3$ in proposition \ref{linethree}.

\subsection{A geometric description of polynomial bireducibility} \label{170920221620}
In this section we describe geometrically the relation of polynomial bireducibility in $ \mathcal P_n$, for $n\ge 3$, that is the orbit relation of the action $\alpha $; in other words, we give a geometric description of the elements of $\Sigma_n$.

Fix $B=\{ b_1,\ldots ,b_n\}\in \mathcal P_n$.
Given $(x_1,\ldots ,x_n)\in D(f_n)$, the set $A=\{ x_1,\ldots ,x_n\} $ is in $[B]$ if and only if there exist a linear polynomial $P(X)=cX+c'$ and a permutation $\sigma $ of $\{ 1,\ldots ,n\} $ such that $P(x_j)=b_{\sigma (j)}$, for $j\in\{ 1,\dots ,n\} $.
Therefore, the condition is that there exists $\sigma\in Sym_n$ such that $(x_1,\ldots ,x_n)$ makes the linear system whose augmented matrix is $N= \begin {pmatrix} x_1 &  1 & b_{\sigma (1)} \\
x_2 &  1 & b_{\sigma (2)} \\
\ldots & \ldots & \ldots \\
x_n &  1 & b_{\sigma (n)} \end{pmatrix} $ compatible.
As $rk \begin{pmatrix}
x_1 & 1 \\
x_2 & 1 \\
\ldots & \ldots \\
x_n & 1
\end{pmatrix} =2$, the condition becomes
\begin{equation} \label{eqnrank}
rkN=2,
\end{equation}
which means that the points $(1,\ldots ,1),(x_1,\ldots ,x_n),(b_{\sigma (1)},\ldots ,b_{\sigma (n)})$ belong to the same plane through the origin.

More precisely, since any two rows of $N$ are linearly independent, \eqref{eqnrank} means that every row is linearly dependent from the first two rows, which translates into the system of linear equations:
\begin{equation} \label{planesn}
\left \{ \begin{array}{lcl} (b_{\sigma (3)}-b_{\sigma (2)})x_1+(b_{\sigma (1)}-b_{\sigma (3)})x_2+ (b_{\sigma (2)}-b_{\sigma (1)})x_3 & = & 0\\
\ldots & & \\
                      (b_{\sigma (j)}-b_{\sigma (2)})x_1+(b_{\sigma (1)}-b_{\sigma (j)})x_2+ (b_{\sigma (2)}-b_{\sigma (1)})x_j & = & 0\\
\ldots & & \\
                        (b_{\sigma (n)}-b_{\sigma (2)})x_1+(b_{\sigma (1)}-b_{\sigma (n)})x_2+ (b_{\sigma (2)}-b_{\sigma (1)})x_n & = & 0 \end{array} \right .
.
\end{equation}
As the coefficient matrix of \eqref{planesn} has rank $n-2$, the set $\pi_{\sigma }$ of the solutions to this system in $k^n$ is a $2$-dimensional vector space, that is a plane, containing the line
\begin{equation} \label{liner}
r:\quad x_1=x_2=\ldots =x_n.
\end{equation}
Notice that $\pi_{\sigma }$ also depends on the enumeration of $B$, while $ \mathcal F_B=\{\pi_{\sigma }\}_{\sigma\in Sym_n}$ does not ---see definition \ref{defexcptl}.

Using the above discussion and notation, the following is obtained.

\begin{lemma}
Fix $\sigma\in Sym_n$ and let $B=\{ b_1,\ldots ,b_n\}\in \mathcal P_n$.
Let $\pi_{\sigma }$ and $r$ be defined as in \eqref{planesn} and \eqref{liner}, respectively.
If $A\in \mathcal P_n$ and $A \pol B$, then there exists $(a_1,\ldots ,a_n)\in\pi_{\sigma }\setminus r$ such that $A=\{ a_1,\ldots ,a_n\} $.
Conversely, if $(a_1,\ldots ,a_n)\in\pi_{\sigma }\setminus r$ then $\{ a_1,\ldots ,a_n\} \pol B$.
\end{lemma}

\begin{proof}
Assume that the polynomial $P$ reduces $A=\{ x_1,\ldots ,x_n\} $ to $B$, so that $P$ is linear by proposition \ref{linearreduction}.
Then there exists $\tau\in Sym_n$ such that $\forall j\in\{ 1,\ldots ,n\} ,P(x_j)=b_{\tau (j)}$, whence $\forall j\in\{ 1,\ldots ,n\} ,P(x_{\tau^{-1}\sigma (j)})=b_{\sigma (j)}$.
Therefore $(x_{\tau^{-1}\sigma (1)},\ldots ,x_{\tau^{-1}\sigma (n)})\in\pi_{\sigma }\setminus r$.

The converse holds by the previous discussion and the fact that the elements in $\pi_{\sigma }\setminus r$ have coordinates that are all distinct.
\end{proof}

Therefore we have the following.

\begin{remark} \label{26102022}
\begin{enumerate}
\item The set of all $(x_1,\ldots ,x_n)$ such that $\{ x_1,\ldots ,x_n\}\in [B]$ is $\bigcup_{\sigma\in Sym_n}\pi_{\sigma }\setminus r=(\theta'\theta )^{-1}([B])$: this is the fiber of $[B]\in\Sigma_n$ under the morphism $\theta'\theta :D(f_n)\to\Sigma_n$.
The closure of this fiber in $k^n$ is $\bigcup_{\sigma\in Sym_n}\pi_{\sigma }$.
\item For any fixed $\sigma\in Sym_n$,
\[
[B]=\{\{ x_1,\ldots ,x_n\}\mid (x_1,\ldots ,x_n)\in\pi_{\sigma }\setminus r\} .
\]
\item Two distinct planes as in \eqref{planesn}, for the same or distinct $B$, intersect in $r$.
\item Letting $B$ range in $ \mathcal P_n$, the planes in \eqref{planesn} cover $k^n$ except for those points having two equal coordinates, that is the points on the $ \frac{n(n-1)}2 $ hyperplanes $x_j=x_{j'}$ for $j\ne j'$.
\end{enumerate}
\end{remark}

\begin{corollary} \label{02102022}
\begin{enumerate}
\item Every polynomial class $[B]$ is a $2$-dimensional subvariety of $ \mathcal P_n$.
\item $\dim (\Sigma_n)=n-2$.
\end{enumerate}
\end{corollary}

\begin{proof}
Every fiber of the morphism $\theta $ is finite and every fiber of the morphism $\theta'\theta $ has dimension $2$ by remark \ref{26102022}(1).
Apply \cite[corollary IV.3.8(2)]{perrin2008}.
\end{proof}

So every polynomial class is identified by any plane $\pi_{\sigma }$ in the set \eqref{planesn} of at most $n!$ planes, where $\sigma $ ranges over $Sym_n$.
Such planes belong to the sheaf of planes based on line $r$.

\begin{definition} \label{defexcptl}
Let $B=\{ b_1,\ldots ,b_n\}\in \mathcal P_n$.
\begin{itemize}
\item We call each of the planes in \eqref{planesn} a \emph{characteristic plane} of $B$, or a \emph{characteristic plane} of $[B]$.
Denote $ \chpl_B$ the set of characteristic planes of $B$.
\item We call \emph{characteristic number} of $B$, or \emph{characteristic number} of $[B]$, and denote it $\chi (B)$ or $\chi ([B])$, the number of characteristic planes of $B$, that is $ \card ( \chpl_B)$.
\end{itemize}
\end{definition}

Therefore $\chi (B)\le n!$ ---see also definition \ref{defexcsetclass}.

\begin{definition} \label{action}
Given $B=\{ b_1,\ldots ,b_n\} $, let $\beta :Sym_n\times \chpl_B\to \chpl_B$ be the action defined by letting $\beta (\tau ,\pi_{\sigma })=\pi_{\tau\sigma }$.
\end{definition}

\begin{lemma}
Action $\beta $ depends only on the polynomial class $[B]$, and not on the representative set $B$ nor on its enumeration.
\end{lemma}

\begin{proof}
First we show that $\beta $ does not depend on any enumeration of $B=\{ b_1,\ldots ,b_n\} $.
So let $B=\{ b_{\rho (1)},\ldots ,b_{\rho (n)}\} $ be another enumeration, with $\rho\in Sym_n$.
For every $\sigma\in Sym_n$, let $\pi'_{\sigma }$ be the plane defined as in \eqref{planesn} using this new enumeration, that is
\[
\pi'_{\sigma }:(b_{\sigma\rho (j)}-b_{\sigma\rho (2)})x_1+(b_{\sigma\rho (1)}-b_{\sigma\rho (j)})x_2+(b_{\sigma\rho (2)}-b_{\sigma\rho (1)})x_j=0,\quad j\in\{ 3,\ldots ,n\} .
\]
Then $\pi'_{\sigma }=\pi_{\sigma\rho }$.
Therefore, if $\beta':(\tau ,\pi'_{\sigma })\mapsto\pi'_{\tau\sigma }$ is the action on $ \chpl_B$ defined using the new enumeration, $\beta (\tau ,\pi'_{\sigma })=\beta (\tau ,\pi_{\sigma\rho })=\pi_{\tau\sigma\rho }=\pi'_{\tau\sigma }=\beta'(\tau ,\pi'_{\sigma })$, whence $\beta =\beta'$.

Now let $C\in [B]$, and let $(c_1,\ldots ,c_n)\in\pi_{id}$ such that $C=\{ c_1,\ldots ,c_n\} $, where $id$ is the identity permutation on $\{ 1,\ldots ,n\} $.
If $\pi'_{\sigma }$ is the plane defined as in \eqref{planesn} using $(c_1,\ldots ,c_n)$, that is
\[
\pi'_{\sigma }:(c_{\sigma (j)}-c_{\sigma (2)})x_1+(c_{\sigma (1)}-c_{\sigma (j)})x_2+(c_{\sigma (2)}-c_{\sigma (1)})x_j=0,\quad j\in\{ 3,\ldots ,n\} ,
\]
it is enough to show that $\pi'_{\sigma }=\pi_{\sigma }$ for every $\sigma\in Sym_n$.
Since $r\subseteq\pi_{\sigma }\cap\pi'_{\sigma }$ and $(c_{\sigma (1)},\ldots ,c_{\sigma (n)})\in\pi'_{\sigma }$ by direct substitution into the equation, it is in turn enough to show that $(c_{\sigma (1)},\ldots ,c_{\sigma (n)})\in\pi_{\sigma }$, that is $(c_1,\ldots ,c_n)\in\pi =\{ (p_1,\ldots ,p_n)\in k^n\mid (p_{\sigma (1)},\ldots ,p_{\sigma (n)})\in\pi_{\sigma }\} $.
But this follows since $\pi $ is a plane including $r$ and containing the point $(b_1,\ldots ,b_n)$, so $\pi =\pi_{id}$.
\end{proof}

Therefore, letting $ \mathcal F $ be the collection of planes containing line $r$ and that are not contained in any hyperplane $x_j=x_{j'}$, for $j\ne j'$, definition \ref{action} yields an action $Sym_n\times \mathcal F \to \mathcal F $, which we still denote $\beta $.
The orbits are the sets of characteristic planes of a given polynomial class.

Given $(b_1,\ldots ,b_n)\in D(f_n)$, the plane in $ \mathcal F $ which $(b_1,\ldots ,b_n)$ belongs to is obtained by choosing in \eqref{planesn} the permutation $\sigma $ to be the identity, that is
\begin{equation} \label{planepi}
\pi_{id}: \left \{ \begin{array}{lcl} (b_2-b_1) (x_3 -x_1) +  (b_1-b_3) (x_2-x_1) & = & 0\\
\ldots & & \\
                         (b_2-b_1) (x_j -x_1) +  (b_1-b_j) (x_2-x_1) & = & 0 \\
\ldots & & \\
                        (b_2-b_1) (x_n -x_1) +  (b_1-b_n) (x_2-x_1) & =& 0
\end{array} \right . .
\end{equation}

For $j\in\{ 1,\ldots ,n\}$ let
\begin{equation} \label{deflambda}
\lambda_j= \frac{b_1-b_j}{b_1-b_2} .
\end{equation}
Notice that $\lambda_1=0,\lambda_2=1$, and the subset of $k^{n-2}$ over which $(\lambda_3,\ldots ,\lambda_n)$ ranges for $(b_1,\ldots ,b_n)\in D(f_n)$ is defined by the conditions
\begin{equation} \label{eqlambdah} \begin{array}{l}
\forall j\in\{ 3,\ldots ,n\} ,\lambda_j\notin\{ 0,1\} \\
\forall j,j'\in\{ 3,\ldots ,n\} ,j\ne j'\Rightarrow\lambda_j\ne\lambda_{j'}
\end{array} . \end{equation}

The subset of all $(\lambda_3,\ldots ,\lambda_n)\in k^{n-2}$ satisfying \eqref{eqlambdah} is a standard open subset of $k^{n-2}$, namely it is $D(S_n)$ where $S_n(\Lambda_3,\ldots ,\Lambda_n)=\prod_{h=3}^n\Lambda_h(\Lambda_h-1)\cdot\prod_{3\le h<h'\le n}(\Lambda_h-\Lambda_{h'})$.
Therefore \eqref{deflambda} defines a surjective morphism $\eta :D(f_n)\to D(S_n)$:
\[
\eta (b_1,\ldots ,b_n)= \left (
\frac{b_1-b_3}{b_1-b_2} ,\ldots , \frac{b_1-b_n}{b_1-b_2}
\right )
.
\]

System \eqref{planepi} can be written as
\begin{equation} \label{sheaves}
(x_j-x_1) + \lambda_j(x_1-x_2) = 0,\qquad \text{for } j\in\{ 3,\ldots ,n\}
\end{equation}
where the relationships between the parameters are given by
\begin{equation} \label{projectivity}
\lambda_j= \frac{b_1-b_j}{b_1-b_{j'}} \lambda_{j'},\qquad \text{for } j,j'\in\{ 3,\ldots ,n\} .
\end{equation}

The following remark gathers some consequences of equations \eqref{sheaves}.

\begin{remark} \label{pointszerone}
\begin{enumerate}
\item Every characteristic plane \eqref{planepi} contains exactly one point with $x_1=0,x_2=1$, namely the point $\lambda =(0,1,\lambda_3,\ldots ,\lambda_n)$.
\item The fibers of the morphism $\eta $ are the characteristic planes, deprived of line $r$: that is, $\eta^{-1}(\{ (\lambda_3,\ldots ,\lambda_n)\} )$ is the characteristic plane, without line $r$, containing the point $(0,1,\lambda_3,\ldots ,\lambda_n)$.
\item Equations \eqref{sheaves} imply also that
\begin{equation} \label{lambdax}
\lambda_j= \frac{x_1-x_j}{x_1-x_2} .
\end{equation}
Therefore the knowledge of $(\lambda_3,\ldots ,\lambda_n)\in D(S_n)$ allows to find all elements $\{ x_1,\ldots ,x_n\} $ of $[\{ 0,1,\lambda_3,\ldots ,\lambda_n\} ]$ by fixing arbitrarily distinct elements $x_1,x_2\in k$ and applying \eqref{lambdax}.
\item The function $(\lambda_3,\ldots ,\lambda_n)\mapsto (0,1,\lambda_3,\ldots ,\lambda_n)$ is a morphism $\xi :D(S_n)\to D(f_n)$, which is a right inverse of $\eta $.
Therefore, setting $\eta'=\theta'\theta\xi :D(S_n)\to\Sigma_n$ we have the following commutative diagram of morphisms between varieties:
\begin{equation} \label{commdiag} \begin{tikzcd}[column sep=2cm]
D(f_n)\arrow[r,"\theta"] \arrow[d,"\eta"] & \mathcal P_n \arrow[d,"\theta'"] \\
D(S_n)\arrow[r,"\eta'"] & \Sigma_n
\end{tikzcd} \end{equation}
Given $\Theta\in\Sigma_n$, the fiber $\eta^{\prime -1}(\{\Theta\} )$ is a set of $\chi (\Theta )$ elements.

All these morphisms depend on $n$.
When it is important to specify the value of $n$, this will be added as a subscript.
\item From the commutative diagram \eqref{commdiag} it also follows that a subset $F\subseteq\Sigma_n$ is closed if and only if $\eta^{\prime -1}(F)$ is closed in $D(S_n)$.
In other words, the topology on $\Sigma_n$ is homeomorphic to the quotient topology of $D(S_n)$ with respect to the equivalence relation induced by $\eta'$.
\end{enumerate}
\end{remark}

The definition \eqref{deflambda} of the coefficients $(\lambda_3,\ldots ,\lambda_n)$ can be used to provide complete invariants for the relation of polynomial bireducibility on finite sets.
More precisely, if in addition to the Zariski topology $k$ also carries a Polish topology with respect to which the operations are Borel, like in the case $k= \C $, the discussion above can be reframed in terms of the classification of equivalence relations under Borel reducibility ---a general reference for the subject is \cite{gao2009}.
Notice that in this case $Fin=\bigcup_{n\in \N } \mathcal P_n$ is a $F_{\sigma }$ subset of $K(k)$, the Polish space of compact subsets of $k$ endowed with the Vietoris topology, since each $\bigcup_{h\le n} \mathcal P_h$ is a closed set.
Therefore $Fin$ is a standard Borel space.
The following proposition shows that from a descriptive set theoretic standpoint polynomial equivalence on finite sets is quite a simple equivalence relation.

\begin{proposition}
If $k$ is endowed with a Polish topology with respect to which the operations are Borel, then the restriction of the equivalence relation $ \eqpol $ to $Fin$ is a smooth equivalence relation.
\end{proposition}

\begin{proof}
It is enough to prove that for each $n\ge 3$ the restriction $ \eqpol^n$ of $ \eqpol $ to $ \mathcal P_n$ is smooth.
For this, a Borel function $\Phi : \mathcal P_n\to k^{n-2}$ is defined such that $A \eqpol^nB\Leftrightarrow\Phi (A)=\Phi (B)$.
Fix a Borel total order $\trianglelefteq$ of $k^{n-2}$.

By \cite[exercise 18.15]{kechri1995} let $\Phi_1,\ldots ,\Phi_n: \mathcal P_n\to k$ be Borel and such that $\forall A\in \mathcal P_n,A=\{\Phi_1(A),\ldots ,\Phi_n(A)\} $.
Let $\Lambda (b_1,\ldots ,b_n)=(\lambda_3,\ldots ,\lambda_n)$ be defined using \eqref{deflambda}.
Then for every $A\in \mathcal P_n$ define $\Phi (A)$ as the $\trianglelefteq$-least of all $\Lambda (\Phi_{\sigma (1)}(A),\ldots ,\Phi_{\sigma (n)}(A))$ for $\sigma\in Sym_n$.
Function $\Phi $ has the desired properties.
\end{proof}

A detailed investigation of the descriptive set theoretic properties of polynomial reducibility and polynomial equivalence is not in the scope of this paper.
We point out nevertheless the following question.

\begin{question}
What is the descriptive set theoretic complexity of $ \eqpol , \pol $ on $K(k)$?
\end{question}

\subsection{$\Sigma_3$ is a line} \label{170920221621}
Let $R_n^{\gamma }$ be the subspace of $k[X_1,\ldots ,X_n]$ consisting of all polynomials that are symmetric, homogeneous of degree $\gamma $, and translation invariant (including the zero polynomial).
By \cite{liptra2008}, a basis of the vector space $R_n^{\gamma }$ is given by the polynomials
\[
W_I(X_1,\ldots ,X_n)=\prod_{\ell =2}^n \left [
\sum_{j=1}^n \left (
X_j- \frac 1n \sum_{l=1}^nX_l
\right )^{\ell }
\right ]^{i_{\ell }}
\]
where $I$ is a partition of $\gamma $ into integers from $2$ to $n$ and $i_{\ell }$ is the multiplicity of $\ell $ in $I$.
Letting
\begin{equation} \label{basisw}
W_{\ell }=\sum_{j=1}^n \left (
X_j- \frac 1n \sum_{l=1}^nX_l
\right )^{\ell },
\end{equation}
the elements of such a basis are all polynomials of the form
\begin{equation*}
W_2^{i_2}W_3^{i_3}\cdot\ldots\cdot W_n^{i_n}
\end{equation*}
where $2i_2+3i_3+\ldots +ni_n=\gamma $.

We now apply this fact and the above discussion to answer question \ref{conjnnmt} for $n=3$.

\begin{theorem} \label{linethree}
$\Sigma_3$ is isomorphic to an affine line.
\end{theorem}

\begin{proof}
We show that the algebra $\Gamma (\Sigma_3)$ is isomorphic to the algebra of regular functions of an affine line.

By lemma \ref{lemnnrgamma}, if $F\in\Gamma (\Sigma_3)$ then $F= \frac{g(X_1,X_2,X_3)}{[f_3(X_1,X_2,X_3)]^{\gamma }} $ where $g$ is symmetric, homogeneous of degree $6\gamma $, and translation invariant.
By applying \eqref{basisw}, we can write
\[
F= \frac{\sum_{2i_2+3i_3=6\gamma }\beta_{i_2i_3}W_2^{i_2}W_3^{i_3}}{f_3^{\gamma }}
\]
where $f_3=\sum_{2j_2+3j_3=6}\alpha_{j_2j_3}W_2^{j_2}W_3^{j_3}$; since $2j_2+3j_3=6$ means that either $j_2=3,j_3=0$ or $j_2=0,j_3=2$, we have $f_3=\alpha_0W_2^3+\alpha_1W_3^2$, for some fixed coefficients $\alpha_0,\alpha_1$ ---which are not both null.

The relation $2i_2+3i_3=6\gamma $ implies that $i_3$ is even, so let $i_3=2i'_3$ with $0\le i'_3\le\gamma $, whence $i_2=3\gamma -3i'_3$.
Setting $T_2=W_2^3,T_3=W_3^2$, we get
\begin{equation} \label{eqnthreecompl}
F= \frac{\sum_{i'_3=0}^{\gamma }\beta_{i'_3}T_2^{\gamma -i'_3}T_3^{i'_3}}{(\alpha_0T_2+\alpha_1T_3)^{\gamma }} .
\end{equation}
The algebra of all functions as in \eqref{eqnthreecompl} is the algebra of regular functions of the complement of the singleton $\{ (\alpha_1:-\alpha_0)\} $ in the projective line $ \mathbb P^1$, that is an affine line.
\end{proof}

\subsection{Exceptional classes} \label{excclasses}

\begin{definition} \label{defexcsetclass}
Let $n\ge 3$ and let $B\in \mathcal P_n$.
If $\chi (B)<n!$ we say that $B$ is an \emph{exceptional set} and $[B]$ is an \emph{exceptional class}.
\end{definition}

So a polynomial class is an exceptional class if and only if there are distinct $\sigma ,\sigma'\in Sym_n$ such that the planes $\pi_{\sigma },\pi_{\sigma'}$ described by the system \eqref{planesn} for $\sigma $ and for $\sigma'$ coincide.

Exceptional classes have a peculiar role with respect to the relation $ \pol $, see theorem \ref{twontwo}.
So the following questions appear to be natural.

\begin{itemize}
\item What are the exceptional classes?
\item Given an exceptional class, what is its characteristic number?
\end{itemize}

The polynomial class $[B]$ is an exceptional class if and only if given some (equivalently, any) $\pi\in \chpl_B$, the stabiliser of $\pi $ under action $\beta $ is not the trivial subgroup of $Sym_n$; equivalently, if and only if there exists a non-identity permutation $\sigma $ such that $\pi_{\sigma }=\pi_{id}$.
Therefore, by \eqref{sheaves}, a polynomial class $[\{ b_1,\ldots ,b_n\} ]$ is an exceptional class if and only if there exists a non-identity permutation $\sigma $ such that, for every $j\in\{ 3,\ldots ,n\} $,
\begin{equation} \label{precrucialrel}
\lambda_j= \frac{b_{\sigma (1)}-b_{\sigma (j)}}{b_{\sigma (1)}-b_{\sigma (2)}} ,
\end{equation}
that is
\begin{equation} \label{precrucialrelbis}
\frac{b_{\sigma (1)}-b_{\sigma (2)}}{b_1-b_2} = \frac{b_{\sigma (1)}-b_{\sigma (j)}}{b_1-b_j} .
\end{equation}
Notice that \eqref{precrucialrel} always holds for $j\in\{ 1,2\} $, while \eqref{precrucialrelbis} always holds for $j=2$.
Note also that equations \eqref{precrucialrel} can be written equivalently as
\begin{equation} \label{crucialrel}
\lambda_{\sigma (j)}-\lambda_{\sigma (1)}=\lambda_j(\lambda_{\sigma (2)}-\lambda_{\sigma (1)}).
\end{equation}

Any enumeration $B=\{ b_1,\ldots ,b_n\} $ induces an isomorphism (which depends on the chosen enumeration)
\begin{equation} \label{06062022833}
Sym(B)\to Sym_n,\qquad\tau\mapsto\sigma
\end{equation}
defined by
\begin{equation} \label{tausigma}
\tau (b_j)=b_{\sigma (j)}.
\end{equation}
The images of a fixed $\tau\in Sym(B)$ under the isomorphisms \eqref{06062022833} induced by different enumerations of $B$ are conjugate.

If $\tau\in Sym(B)$, given distinct $b,b'\in B$ define
\begin{equation} \label{diffq}
Y_{\tau }(b,b')= \frac{\tau (b)-\tau (b')}{b-b'} .
\end{equation}

The following gives a characterisation of exceptional classes in terms of the difference quotient \eqref{diffq}.

\begin{proposition} \label{propexistsforall}
Let $B\in \mathcal P_n$.
The following are equivalent:
\begin{enumerate}
\item The polynomial class $[B]$ is an exceptional class.
\item For any $b \in B$ there exists a non-identity permutation $\tau\in Sym(B)$ such that the difference quotient $Y_{\tau }(b ,b')$ is constant for every $b'\in B\setminus\{ b \} $.
\item There exist $b \in B$ and a non-identity permutation $\tau\in Sym(B)$ such that the difference quotient $Y_{\tau }(b ,b')$ is constant for every $b'\in B\setminus\{ b \} $.
\item There exists a non-identity permutation $\tau\in Sym(B)$ such that the difference quotient $Y_{\tau }(b,b')$ is constant for any distinct $b,b'\in B$.
\item The stabiliser of $B$ under action $\alpha $ of definition \ref{20112022} is not trivial.
\end{enumerate}
\end{proposition}

\begin{proof}
$(1)\Rightarrow (2)$.
Let $B=\{ b_1,\ldots ,b_n\} $, with $b_1=b $.
Therefore \eqref{precrucialrelbis} provides a non-identity $\sigma\in Sym_n$ such that the ratio $ \frac{b_{\sigma (1)}-b_{\sigma (j)}}{b_1-b_j} $ does not depend on $j\ne 1$.
Defining $\tau $ as is \eqref{tausigma} yields the conclusion.

$(4)\Rightarrow (2)\Rightarrow (3)$ is clear.

$(3)\Rightarrow (1)$.
Let $B=\{ b_1,\ldots ,b_n\} $ with $b_1=b $.
Define $\sigma $ as is \eqref{tausigma}.
Then \eqref{precrucialrelbis} holds.

$(3)\Rightarrow (4)$.
Let $\tau ,b$ satisfy (3).
Given distinct $b',b''\in B\setminus\{ b\} $, let $\varepsilon =Y_{\tau }(b,b')=Y_{\tau }(b,b'')$.
Therefore, $\tau (b)-\tau (b')=\varepsilon (b-b')$, whence $\tau (b)-\varepsilon b=\tau (b')-\varepsilon b'$, and similarly $\tau (b)-\varepsilon b=\tau (b'')-\varepsilon b''$.
It follows that $\tau (b')-\varepsilon b'=\tau (b'')-\varepsilon b''$, and finally $Y_{\tau }(b',b'')=\varepsilon $.

$(3)\Rightarrow (5)$.
Let $\tau ,b$ satisfy (3), and set $\varepsilon =Y_{\tau }(b,b')$, for any $b'\ne b$.
Define
\begin{equation} \label{defpcrit}
P(X)=\varepsilon X-\varepsilon b+\tau (b).
\end{equation}
Then $P(B)=B$, since
\begin{equation} \label{equpolynbb}
\forall x\in B,P(x)=\tau (x).
\end{equation}

$(5)\Rightarrow (4)$.
If $P\in \mathcal L \setminus\{ X\} $ is such that $P(B)=B$, then the restriction of $P$ to $B$ determines a non-identity permutation $\tau $ of $B$; the value of the difference ratios $Y_{\tau }(b,b')$ is constant and coincide with the coefficient of degree $1$ of $P$.
\end{proof}

In the proof of $(3)\Rightarrow (5)$ of proposition \ref{propexistsforall} the polynomial $P$ does not depend on the choice of $b$, but only on $\tau $.
Indeed, if $P_{b_0},P_{b_1}$ are the polynomials defined by choosing $b_0,b_1$, respectively, from \eqref{equpolynbb} it follows $P_{b_0}=P_{b_1}$ since both $P_{b_0},P_{b_1}$ are linear and coincide on $n$ points.

\begin{definition} \label{defpitau}
Let $B\in \mathcal P_n,\tau\in Sym(B)$.
Then $\tau $ is a \emph{characteristic permutation} of $B$ if $Y_{\tau }(b,b')$ is constant, for $b\ne b'$; in this case, denote $Y_{\tau }$ such a value.
Also, let $G_B$ be the set of characteristic permutations of $B$.
For $\tau\in G_B$, let $P_{\tau }$ be the polynomial defined as in \eqref{defpcrit}, and set $H_B=\{ P_{\tau }\mid\tau\in G_B\} $.
\end{definition}

\begin{proposition} \label{charactzero}
Let $B\in \mathcal P_n$.
\begin{enumerate}
\item $G_B$ is a subgroup of $Sym(B)$, $H_B$ is the stabiliser of $B$ under action $\alpha $, and the function $\tau\mapsto P_{\tau }$ is an isomorphism between them.
\item The function $Y:\tau\mapsto Y_{\tau }$ is an injective morphism of the group $G_B$ into the multiplicative group $k\setminus\{ 0\} $.
\end{enumerate}
\end{proposition}

\begin{proof}
(1)
\begin{itemize}
\item The identity $id$ is a characteristic permutation, since $Y_{id}(b,b')=1$ for $b\ne b'$.
\item Let $\tau ,\tau'\in G_B$.
If $b\ne b'$, then $Y_{\tau\tau'}(b,b')= \frac{\tau\tau'(b)-\tau\tau'(b')}{\tau'(b)-\tau'(b')} \frac{\tau'(b)-\tau'(b')}{b-b'} =Y_{\tau'}Y_{\tau }$, so $\tau\tau'\in G_B$.
\item Let $\tau\in G_B$.
If $b\ne b'$, then $Y_{\tau^{-1}}(b,b')= \frac{\tau^{-1}(b)-\tau^{-1}(b')}{\tau\tau^{-1}(b)-\tau\tau^{-1}(b')} =Y_{\tau }^{-1}$, so $\tau^{-1}\in G_B$.
\end{itemize}

The fact that every $P_{\tau }$ is in the stabiliser of $B$ follows immediately from \eqref{equpolynbb}, while the fact that every element in the stabiliser of $B$ is of the form $P_{\tau }$ for some $\tau\in G_B$ is argued as in the proof of $(5)\Rightarrow (4)$ of proposition \ref{propexistsforall}.
The fact that $\tau\mapsto P_{\tau }$ is a morphism follows from \eqref{equpolynbb}, since
\[
\forall x\in B,P_{\tau\tau'}(x)=\tau\tau'(x)=P_{\tau }P_{\tau'}(x)
\]
whence $P_{\tau\tau'}=P_{\tau }P_{\tau'}$.
Again from \eqref{equpolynbb}, if $P_{\tau }(X)=X$, than $\tau $ is the identity, so $\tau\mapsto P_{\tau }$ is an isomorphism.

(2) By the first part of the proof of (1), function $Y$ is a morphism.
To see that it is injective, let $\tau\in G_B$ with $Y_{\tau }=1$.
Then the value $\tau (b)-b$ is constant over $B$, say $\tau (b)=b+c$ for some $c\in k$.
Therefore $\tau^r(b)=b+rc$ for every $r$.
Let $m$ be the order of $\tau $.
It follows that $b=\tau^m(b)=b+mc$, whence $c=0$, so that $\tau $ is the identity.
\end{proof}

Therefore $B$ is an exceptional set if and only if $G_B$ is not the trivial subgroup of $Sym(B)$.

\begin{remark} \label{sensstab}
Given characteristic planes $\pi ,\pi'$ of $B$, their stabilisers $Stab_{\pi }^{\beta },Stab_{\pi'}^{\beta }$ are conjugate to each other, hence they are isomorphic.

Note that, by the proof of proposition \ref{propexistsforall}, $\tau\in Sym(B)$ is a characteristic permutation of $B$ if and only if any (equivalently, every) permutation $\sigma $ associated to $\tau $ by \eqref{tausigma} is in the stabiliser of some characteristic plane of $B$.
Therefore facts on characteristic permutations give rise to corresponding results on permutations stabilising a characteristic plane.
For instance, the polynomial class $[B]$ is an exceptional class if and only if for some (equivalently, all) $\pi\in \chpl_B$ the stabiliser $Stab_{\pi }^{\beta }$ is not the trivial subgroup of $Sym_n$; in fact, $\chi (B)= \frac{n!}{ \card (Stab_{\pi }^{\beta })} = \frac{n!}{ \card (G_B)} = \frac{n!}{ \card (H_B)} $.
\end{remark}

\begin{lemma} \label{notwo}
If $\tau\in G_B$ and there are $b,b'$ with $b\ne b'$ such that $\tau (b)=b,\tau (b')=b'$, then $\tau $ is the identity.
\end{lemma}

\begin{proof}
By the assumption, $Y_{\tau }= \frac{\tau (b)-\tau (b')}{b-b'} =1$.
Now apply proposition \ref{charactzero}(2).
\end{proof}

Every non-trivial finite group has an element whose order is prime.
Therefore, to identify exceptional sets it is enough to consider permutations having as order a prime number.

\begin{corollary} \label{multofprime}
The set $B\in \mathcal P_n$ is an exceptional set if and only if there exists a characteristic permutation $\tau $ of $B$ such that
\[
\tau  =c_1\cdot\ldots\cdot c_s
\]
where $c_1,\ldots ,c_s$ are disjoint $p$-cycles for some prime number $p$, and either $n=sp$ or $n=sp+1$.
\end{corollary}

\begin{proof}
Assume $B$ is an exceptional set.
Let $\tau\in G_B$ have order some prime number $p$.
Then $\tau =c_1\cdot\ldots\cdot c_s$ for some disjoint $p$-cycles $c_1,\ldots ,c_s$.
The fact that either $n=sp$ or $n=sp+1$ follows from lemma \ref{notwo}.

The converse is immediate.
\end{proof}

\begin{lemma} \label{lemrroot}
Let $B\in \mathcal P_n$, and let $\tau=c_1\cdot\ldots\cdot c_s\in G_B\setminus\{ id\} $, where $c_1,\ldots ,c_s$ are pairwise disjoint cycles; let $r\ge 2$ be the order of at least one of them.
Then $Y_{\tau }$ is a primitive $r$-th root of unity.
\end{lemma}

\begin{proof}
Let $(b_1\ \ldots\ b_r)$ be a $r$-cycle of $\tau $.
Let $P_{\tau }(X)=Y_{\tau }X-Y_{\tau }b_1+\tau (b_1)$ be defined as in \eqref{defpcrit}, choosing $b_1$ as $b$.
The restriction of $P_{\tau }^r$ to $\{ b_1,\ldots ,b_r\} $ is the identity so, since $r\ge 2$, the polynomial $P_{\tau }^r$ is the identity and, by \eqref{2012202218}, $Y_{\tau }^r=1$; moreover
\begin{equation} \label{201220222226}
(Y_{\tau }^{r-1}+\ldots +Y_{\tau }+1)(-Y_{\tau }b_1+\tau (b_1))=0.
\end{equation}
If $Y_{\tau }=1$, from \eqref{201220222226} it would follow $\tau (b_1)=b_1$, which is false.
Therefore $Y_{\tau }\ne 1$, and since $P_{\tau }^j$ is not the identity for any $j<r$, from \eqref{201220221809} it follows that $Y_{\tau }^j\ne 1$.
\end{proof}

\begin{corollary} \label{cormnd}
If $\tau $ is a non-identity characteristic permutation of $B\in \mathcal P_n$, then $\tau =c_1\cdot\ldots\cdot c_s$ is the product of the disjoint cycles $c_1,\ldots ,c_s$, where all such cycles have the same order, say $r$, and either $n=sr$ or $n=sr+1$.
\end{corollary}

\begin{proof}
By lemmas \ref{notwo} and \ref{lemrroot}.
\end{proof}

\begin{proposition} \label{corabel}
The group $G_B$ is a cyclic group.
Its order is either a divisor of $n$ or a divisor of $n-1$.
\end{proposition}

\begin{proof}
By proposition \ref{charactzero}(2) and lemma \ref{lemrroot}, $G_B$ is isomorphic to a finite subgroup of the group of roots of unity; in particular $G_B$ is cyclic.
Let $r$ be its order, and let $\tau\in G_B$ have order $r$.
By corollary \ref{cormnd}, $\tau $ is a product of disjoint $r$-cycles, so the conclusion follows by applying corollary \ref{cormnd} again.
\end{proof}

We now obtain a characterisation of the exceptional sets.

\begin{definition}
Let $r\ge 2$.
A \emph{regular star} $r$-\emph{gon} is a sequence $(z_1,\ldots ,z_r)$ of distinct elements of $k$, called \emph{vertices}, such that there exists $P\in \mathcal L $ with the property:
\begin{equation} \label{stardef}
P(z_1)=z_2,\quad P(z_2)=z_3,\quad\ldots ,\quad P(z_{r-1})=z_r,\quad P(z_r)=z_1.
\end{equation}
A \emph{regular} $r$-\emph{gon} is the set of vertices of a regular star $r$-gon.
\end{definition}

So, if $k= \C $, regular star $r$-gons and regular $r$-gons are the usual geometric concepts in the Gauss plane.

Note also that if $P$ is as in \eqref{stardef} then $P^r(X)=X$, where $P^r$ denoted the composition of $P$ with itself $r$ times.

\begin{lemma} \label{30112022}
Let $(z_1,\ldots ,z_r)$ be a regular star $r$-gon, and let $P$ be as in \eqref{stardef}.
Then $P(X)=\varepsilon X+c$ where:
\begin{itemize}
\item $\varepsilon $ is a primitive $r$-th root of unity;
\item the fixed point of $P$, namely $ \frac c{1-\varepsilon }$, is the barycentre $b$ of the regular star $r$-gon.
\end{itemize}
\end{lemma}

\begin{proof}
Using \eqref{2012202218}, from $P^r(X)=X$ it follows that $\varepsilon $ is a $r$-th root of unity.
Moreover $\varepsilon\ne 1$, since otherwise $P^r(z_1)=z_1+rc\ne z_1$.
If there existed $r'<r$ such that $\varepsilon^{r'}=1$ then, by \eqref{201220221809}, $P^{r'}(X)=X+c \frac{\varepsilon^{r'}-1}{\varepsilon -1} =X$, a contradiction with \eqref{stardef}.
Finally,
\begin{multline*}
b= \frac{\sum_{j=0}^{r-1}P^j(z)}r= \frac{\sum_{j=0}^{r-1}\varepsilon^jz+\sum_{j=1}^{r-1}\sum_{\ell =0}^{j-1}\varepsilon^{\ell }c}r
= \frac{ \frac{\varepsilon^r-1}{\varepsilon -1} z+\sum_{j=1}^{r-1} \frac{\varepsilon^j-1}{\varepsilon -1} c}r = \\
= \frac{(\varepsilon^{r-1}+\ldots +\varepsilon +1-r)c}{r(\varepsilon -1)} = \frac{c}{1-\varepsilon } .
\end{multline*}
\end{proof}

Exceptional sets turn out to be the unions of regular $r$-gons with the same barycentre, possibly with the barycentre itself.

\begin{theorem} \label{chrexcsets}
Let $B\in \mathcal P_n$.
Then $B$ is an exceptional set if and only if there are $r\ge 2$ and $s$ such that
\begin{itemize}
\item[(a)] either $sr=n$ and $B=\bigcup_{h=1}^sC_h$ where the sets $C_h$ are parwise disjoint, regular $r$-gons with the same barycentre $b$;
\item[(b)] or $sr+1=n$ and $B=\bigcup_{h=1}^sC_h\cup\{ b\} $ where the sets $C_h$ are pairwise disjoint, regular $r$-gons having the same barycentre $b$.
\end{itemize}
\end{theorem}

\begin{proof}
Assume that (a) or (b) holds.
For every $h$, let $P_h(X)=\varepsilon_hX+c_h$ be the linear polynomial as in \eqref{stardef} for a regular $r$-star formed by the elements of $C_h$.

Now fix any $h$.
By lemma \ref{30112022}, for every $h'$ there exists $\ell $ such that $\varepsilon_h=\varepsilon_{h'}^{\ell }$.
Since $ \frac{c_{h'}}{1-\varepsilon_{h'}} = \frac{c_h}{1-\varepsilon_h} =b$, it follows that $c_h=c_{h'} \frac{1-\varepsilon_{h'}^{\ell }}{1-\varepsilon_{h'}} =c_{h'}(1+\varepsilon_{h'}+\varepsilon_{h'}^2+\ldots +\varepsilon_{h'}^{\ell -1})$.
Therefore
\[
P_h(X)=\varepsilon_{h'}^{\ell }X+c_{h'}(1+\varepsilon_{h'}+\varepsilon_{h'}^2+\ldots +\varepsilon_{h'}^{\ell -1})=P_{h'}^{\ell }(X).
\]
Since $P_{h'}$ sends $C_{h'}$ onto itself, the same is true for $P_h$.
As this holds for every $h'$, it follows that $P_h(B)=B$, so the stabiliser of $B$ under the action $\alpha $ is not trivial.

Conversely assume that $B$ is an exceptional set, let $\tau $ be a non-identity characteristic permutation of $B$, and let $P_{\tau }$ be defined by \eqref{defpcrit}.
Let $r$ be the order of $\tau $.
Then $\tau =c_1\cdot\ldots\cdot c_s$ with $sr=n$ or $sr+1=n$ by corollary \ref{cormnd}, where the cycles $c_j$ are pairwise disjoint and their elements form regular $r$-stars, as witnessed by the polynomial $P_{\tau }$ by \eqref{equpolynbb}.
Therefore their barycentres must coincide with the unique fixed point of $P_{\tau }$, which exists since the coefficient $\varepsilon $ in \eqref{defpcrit} is different from $1$ as $\tau $ is not the identity.
\end{proof}

\begin{remark} \label{301120221927}
Theorem \ref{chrexcsets} tells that an exceptional set $B\in \mathcal P_n$ can identified by the choice of a number $s$ as in (a) or (b) of the theorem and the choice of $s+1$ points, namely one point from each regular $r$-gon and the common barycentre.
In turn, by lemma \ref{30112022}, the barycentre can be recovered by knowing the second vertex of one of the regular star $r$-gons and the leading coefficient (which is one of the finitely many primitive $r$-th roots of unity) of the polynomial $P$ of \eqref{stardef}.

So $B$ is identified by $s+1$ parameters.
This idea will be exploited in the computation of the dimension of the set of exceptional classes.
\end{remark}

The following is a converse of proposition \ref{corabel}.

\begin{corollary} \label{orderexcp}
If $r$ is a divisor of $n$ or of $n-1$, then there exists $B\in \mathcal P_n$ such that $G_B$ has order $r$.
\end{corollary}

\begin{proof}
Assume first that $r$ is a divisor of $n$, and let $s$ be such that $sr=n$.
Let $B=\bigcup_{h=1}^sC_h$ where:
\begin{itemize}
\item[(i)] each $C_h$ is a regular $r$-gon;
\item[(ii)] all $C_h$ have $0$ as barycentre;
\item[(iii)] $h\in C_h$ for every $h$.
\end{itemize}
Therefore $C_h=\{ h,\varepsilon h,\varepsilon^2h,\ldots ,\varepsilon^{r-1}h\} $, where $\varepsilon $ is a primitive $r$-th root of unity.

By theorem \ref{chrexcsets}(a), $B$ is an exceptional set.
Let $\tau $ be a generator of $G_B$.
Then the elements in each cycle of $\tau $ are contained in some $C_h$, so the order of $\tau $ is a divisor of $r$; moreover, if $P_h$ is as in the proof of theorem \ref{chrexcsets}, the restriction of $P_h$ to ${C_h}$ has order $r$ and is a power of $\tau |_{C_h}$, so the order of $\tau $ cannot be less than $r$.

A similar argument works if $r$ is a divisor of $n-1$, including in the set $B$ the common barycentre $0$ of the regular $r$-gons.
\end{proof}

Using the isomorphisms \eqref{tausigma}, from the proof of corollary \ref{orderexcp} one obtains also the following.

\begin{corollary} \label{prescribedform}
The set of $\sigma\in Sym_n$ associated via \eqref{tausigma} to some characteristic permutation are all permutations of $\{ 1,\ldots ,n\} $ of the form $\sigma =c_1\cdot\ldots\cdot c_s$ for some disjoint cycles $c_1,\ldots ,c_s$ all of the same length $r$, with $n=sr$ or $n=sr+1$.
\end{corollary}

\begin{corollary} \label{hbgeom}
The set of all characteristic numbers of sets in $ \mathcal P_n$ is the set of all numbers of the form $ \frac{n!}r $ where $r$ ranges over the set of divisors of either $n-1$ or of $n$.
\end{corollary}

\begin{proof}
By proposition \ref{corabel}, corollary \ref{orderexcp}, and the fact that $\chi (B)= \frac{n!}{ \card (G_B)} $.
\end{proof}

\begin{corollary} \label{witntwo}
Let $B\in \mathcal P_n$, and let $\tau\in Sym(B)$ have order $2$.
Then $\tau $ is a characteristic permutation of $B$ if and only if the value of $b+\tau (b)$ is constant for all $b\in B$.
\end{corollary}

\begin{proof}
Lemma \ref{lemrroot} states that $\tau $ is a characteristic permutation of $B$ if and only if $\forall b,b'\in B,b\ne b'\Rightarrow Y_{\tau }(b,b')=-1$; this equality is equivalent to $b+\tau (b)=b'+\tau (b')$.
\end{proof}

\begin{definition}
\begin{itemize}
\item We denote $ \mathcal E_n$ the set of exceptional classes of $\Sigma_n$.
\item We denote $ \mathcal E_n^*$ the set of exceptional classes of $\Sigma_n$ whose elements have characteristic group of even order.
\end{itemize}
\end{definition}

\begin{proposition} \label{170920221714}
$ \mathcal E_n, \mathcal E_n^*$ are subvarieties of $\Sigma_n$.
\end{proposition}

\begin{proof}
Using the notation of remark \ref{pointszerone}(4), a point $(\lambda_3,\ldots ,\lambda_n)\in D(S_n)$ belongs to $\eta^{\prime -1}( \mathcal E_n)$ if and only if there is some non-identity permutation $\sigma $ such that the polynomial system \eqref{crucialrel} is satisfied.
Therefore $\eta^{\prime -1}( \mathcal E_n)$ is the union, with $\sigma $ ranging in $Sym_n\setminus\{ id\} $, of the affine varieties generated by the systems \eqref{crucialrel}, so it is an affine variety invariant with respect to the fibers of $\eta'$.
Consequently, $ \mathcal E_n$ is an affine variety.

The argument for $ \mathcal E_n^*$ is similar, by considering only permutations $\sigma $ of even order.
\end{proof}

To determine the dimension of $ \mathcal E_n, \mathcal E_n^*$, we use that
\begin{equation} \label{13082022}
\dim ( \mathcal E_n)=\dim ((\theta'\theta )^{-1}( \mathcal E_n))-2
\end{equation}
and similarly for $ \mathcal E_n^*$.
So let $I$ be the set of all natural numbers $r$ such that $r\ge 2$ and $r$ is a divisor of either $n-1$ or $n$, and let $I^*$ be the set of even elements of $I$.
For $r\in I$, let $E_r$ be the set of all $(b_1,\ldots ,b_n)\in D(f_n)$ such that $\{ b_1,\ldots ,b_n\} $ satisfies either (a) or (b) of theorem \ref{chrexcsets}.
Therefore
\[ \begin{array}{l}
(\theta'\theta )^{-1}( \mathcal E_n)=\bigcup_{r\in I}E_{r} \\
(\theta'\theta )^{-1}( \mathcal E_n^*)=\bigcup_{r\in I^*}E_r
\end{array}
.
\]

\begin{lemma} \label{keylemdim}
Let $n=sr$ or $n=sr+1$, for some $r\in I$.
Then $E_r$ is a subvariety of $D(f_n)$ and $\dim (E_{r})=s+1$.
\end{lemma}

\begin{proof}
We exploit the idea of remark \ref{301120221927}.

Assume first that $n=sr$.
If $\varepsilon $ is a primitive $r$-th root of unity, consider $\psi_{\varepsilon }:k^{s+1}\to k^n$ defined as follows: given $(b_{1},b_{2},b_{{r+1}},b_{{2r+1}},b_{{3r+1}},\ldots ,b_{{(s-1)r+1}})\in k^{s+1}$, let
\[
\psi_{\varepsilon }(b_{1},b_{2},b_{{r+1}},b_{{2r+1}},b_{{3r+1}},\ldots ,b_{{(s-1)r+1}})=(b_h)_{h=1}^n
\]
where
\begin{equation} \label{291020221132}
b_{{ur+1+j}}=\varepsilon^jb_{{ur+1}}+(1+\varepsilon +\varepsilon^2+\ldots +\varepsilon^{j-1})(b_{2}-\varepsilon b_1)
\end{equation}
for $u\in\{ 0,\ldots ,s-1\} ,j\in\{ 1,\ldots ,r-1\} $.
Therefore $\psi_{\varepsilon }$ takes as input a sequence of $s$ vertices $ \vec b =(b_{1},b_{{r+1}},b_{{2r+1}},\ldots ,b_{{(s-1)r+1}})$ plus an extra vertex $b_{2}$, and outputs $s$ concentric regular star $r$-gons each one starting at a vertex $b_{ur+1}$ in $ \vec b $ and having $\varepsilon b_{ur+1}+b_2-\varepsilon b_1$ as second vertex.

By \eqref{precrucialrelbis} and lemma \ref{lemrroot}, every element of $E_r$ is in the range of some $\sigma\psi_{\varepsilon }$, for $\varepsilon $ a primitive $r$-th root of unity and $\sigma $ a permutation of the coordinates..
Moreover, $\psi_{\varepsilon }(b_{1},b_{2},b_{{r+1}},b_{{2r+1}},\ldots ,b_{{(s-1)r+1}})$ is in $E_r$ if and only if its components $b_h$, for $h\in\{ 1,\ldots ,n\} $, are all distinct, which is equivalent to $(b_{1},b_{2},b_{{r+1}},b_{{2r+1}},\ldots ,b_{{(s-1)r+1}})\in D(g_{\varepsilon })$ where $g_{\varepsilon }$ is the product of the polynomials:
\begin{itemize}
\item $X_{j}-X_{{j'}}$ for $j,j'\in\{ 1,2,r+1,2r+1,3r+1,\ldots ,(s-1)r+1\} ,j\ne j'$,
\item $X_{{vr+1}}-\varepsilon^jX_{{ur+1}}-(1+\varepsilon +\varepsilon^2+\ldots +\varepsilon^{j-1})(X_{2}-\varepsilon X_{1})$ for $u\ne v,j\in\{ 1,\ldots ,r-1\} $,
\item $X_{j}- \frac 1{1-\varepsilon } (X_{2}-\varepsilon X_{1})$ for $j\in\{ r+1,2r+1,3r+1,\ldots ,(s-1)r+1\} $.
\end{itemize}

By equations \eqref{291020221132} the range of every $\psi_{\varepsilon }|_{D(g_{\varepsilon })}$ is an affine variety and $\psi_{\varepsilon }|_{D(g_{\varepsilon })}$ is an isomorphism onto its range.
Therefore $\dim (E_r)=\dim (D(g_{\varepsilon }))=\dim (k^{s+1})$ for any $\varepsilon $, that is $\dim (E_r)=s+1$.

The case $n=sr+1$ is similar.
If $\varepsilon $ is a primitive $r$-th root of unity, consider $\psi_{\varepsilon }:k^{s+1}\to k^n$ defined as follows: given $(b_{1},b_{2},b_{{r+1}},b_{{2r+1}},b_{{3r+1}},\ldots ,b_{{(s-1)r+1}})\in k^{s+1}$, let
\[
\psi_{\varepsilon }(b_{1},b_{2},b_{{r+1}},b_{{2r+1}},b_{{3r+1}},\ldots ,b_{{(s-1)r+1}})=(b_h)_{h=1}^n
\]
where
\begin{align*}
b_{{ur+1+j}}= & \varepsilon^jb_{{ur+1}}+(1+\varepsilon +\varepsilon^2+\ldots +\varepsilon^{j-1})(b_{2}-\varepsilon b_1) \\
 & \text{for } u\in\{ 0,\ldots ,s-1\} ,j\in\{ 1,\ldots ,r-1\} , \\
b_{n}= & \frac{b_{2}-\varepsilon b_{1}}{1-\varepsilon } .
\end{align*}
Again by \eqref{precrucialrelbis} and lemma \ref{lemrroot}, every element of $E_r$ is in the range of some $\sigma\psi_{\varepsilon }$, for $\varepsilon $ a primitive $r$-th root of unity and $\sigma $ a permutation of the coordinates, and $\psi_{\varepsilon }(b_{1},b_{2},b_{{r+1}},b_{{2r+1}},\ldots ,b_{{(s-1)r+1}})$ is in $E_r$ if and only if $(b_{1},b_{2},b_{{r+1}},b_{{2r+1}},\ldots ,b_{{(s-1)r+1}})\in D(g_{\varepsilon })$ where $g_{\varepsilon }$ is the same as above.
Now the proof can be concluded as in the previous case.
\end{proof}

Then
\begin{equation} \label{maxattnd} \begin{array}{l}
\dim ((\theta'\theta )^{-1}( \mathcal E_n))=\max\{\dim (E_r)\}_{r\in I} \\
\dim ((\theta'\theta )^{-1}( \mathcal E_n^*))=\max\{\dim (E_r)\}_{r\in I^*}
\end{array}
.
\end{equation}
Lemma \ref{keylemdim} says that this maximum is attained when $r$ is small.

\begin{corollary} \label{170920221716}
If $n$ is even then $\dim ( \mathcal E_n)=\dim ( \mathcal E_n^*)= \frac {n-2}2 $, if $n$ is odd then $\dim ( \mathcal E_n)=\dim ( \mathcal E_n^*)= \frac{n-3}2$.
\end{corollary}

\begin{proof}
Apply \eqref{maxattnd}, lemma \ref{keylemdim}, and \eqref{13082022}.
\end{proof}

For example, in the affine line $\Sigma_3$, solving \eqref{crucialrel} where $\sigma $ ranges over transpositions, yields for $\lambda_3$ the values $ \frac 12 ,2,-1$.
As $ \left \{
0,1, \frac 12
\right \}
\eqpol \{ 0,1,2\} \eqpol \{0,1,-1\} $, the subvariety $ \mathcal E_3^*$ consists of one point.
Solving \eqref{crucialrel} where $\sigma $ ranges over cycles of order $3$, yields $\lambda_3= \frac{1\pm \sqrt{3} i}2 $.
As $ \left \{
0,1, \frac{1+ \sqrt{3} i}2
\right \}
\eqpol \left \{
0,1, \frac{1- \sqrt{3} i}2
\right \} $, it follows that $ \mathcal E_3= \left \{ \left [ \left \{
0,1, \frac 12
\right \} \right ] , \left [ \left \{
0,1, \frac{1+ \sqrt{3} i}2
\right \} \right ] \right \} $.

\subsection{Polynomial reducibility on finite sets} \label{18092022}
Let $A=\{ a_1,\ldots ,a_m\}\in \mathcal P_m,B=\{ b_1,\ldots ,b_n\}\in \mathcal P_n$, for some $m>n\ge 2$.
Let $P$ be a polynomial such that $A=P^{-1}(B)$ and let $\gamma =\deg (P)$.
Since every equation $P(x)-b_j=0$ has at most $\gamma $ solutions, it follows that $m\le n\gamma $, that is $\gamma\ge \frac mn $.

On the other hand, if $a$ is a multiple solution of $P(x)-b_j=0$ of multiplicity $r_a>1$, then $a$ is a root of multiplicity $r_a-1$ of the derivative $P'$.
Let $M$ be the set of all multiple solutions of the equations $P(x)-b_j=0$ for $j\in\{ 1,\ldots ,n\} $.
Since $P'$ has degree $\gamma -1$, letting $t=\sum_{a\in M}(r_a-1)$ it follows that $0\le t\le\gamma -1$.
From the fact that $n\gamma -t=m$, it then follows that $n\gamma -(\gamma -1)\le m$, whence $\gamma\le \frac{m-1}{n-1} $.

Summing up,
\begin{equation} \label{conditiong}
\frac mn \le\gamma\le \frac{m-1}{n-1} ,
\end{equation}
that is
\begin{equation} \label{conditionginv}
\gamma (n-1)+1\le m\le \gamma n.
\end{equation}
In particular, a necessary condition for a set in $ \mathcal P_m$ to reduce polynomially to a set in $ \mathcal P_n$ is that:
\begin{equation} \label{between}
\text{there exists a natural number in } \left [
\frac mn , \frac {m-1}{n-1}
\right ]
.
\end{equation}

Let $e$ be the integer part of $ \frac mn $, that is the biggest integer such that $m\ge en$.
Then condition \eqref{between} is equivalent to
\begin{equation} \label{e}
\text{either } m=en,\quad \text{ or } m\ge (e+1)n-e.
\end{equation}
Notice also that \eqref{conditionginv}, which can be rewritten as
\begin{equation} \label{unionintervals}
m\in\bigcup_{\gamma\ge 2}[\gamma (n-1)+1,\gamma n]\cap \N ,
\end{equation}
is certainly satisfied when $m\ge (n-1)^2+1$, since in the union displayed in \eqref{unionintervals} the interval with $\gamma =n-1$ is adjacent to the interval with $\gamma =n$, and for bigger values of $\gamma $ the intervals overlap.

By \eqref{e}, given $B\in \mathcal P_n$ the least $m>n$ for which there may exist $A\in \mathcal P_m$ with $A \pol B$ is $m=2n-1$.
In this case, by \eqref{conditiong} any polynomial reducing $A$ to $B$ has degree $2$.
On the other hand, if there is a polynomial of degree $2$ reducing $A$ to $B$, then either $A\in \mathcal P_{2n-1}$ or $A\in \mathcal P_{2n}$.

The following lemma extends remark \ref{fewpols}.

\begin{lemma} \label{040920221508}
Fix $n,m,A,B$, with $2\le n<m,A\in \mathcal P_m,B\in \mathcal P_n$.
Then there are only finitely many (if any) polynomials $P$ such that $A=P^{-1}(B)$.
\end{lemma}

\begin{proof}
It is enough to show that for any $\gamma $ satisfying \eqref{conditiong} there are only finitely many polynomials $P$ of degree $\gamma $ reducing $A$ to $B$.
Any such $P$ determines a partition $\{ A_b^P\}_{b\in B}$ of $A$ such that $A_b^P=P^{-1}(\{ b\} )$ for every $b\in B$.
Therefore it is enough to show that given a partition $\{ A_b\}_{b\in B}$ of $A$ there are finitely many (in fact, at most one) $P$ of degree $\gamma $ such that $A_b=A_b^P$ for every $b\in B$.
Indeed, if $P,Q$ are polynomials of degree $\gamma $ such that $A_b^P=A_b^Q$ for every $b\in B$, then $P(X)-Q(X)$ is a polynomial of degree at most $\gamma $ taking value $0$ on all elements of $A$.
From $\gamma <m$ it then follows that $P=Q$.
\end{proof}

\begin{lemma} \label{lemmaprivileged}
Let $A,B\subseteq k$.
Assume that there is a polynomial $P$ of degree $\gamma\ge 1$ reducing $A$ to $B$, and fix distinct $a,a'\in A$ and distinct $b,b'\in B$ with $P(a)=b$ and $P(a')=b'$.
Let also $x,x',y,y'\in k$ be such that $x\ne x'$ and $y\ne y'$.
Then there exist $A',B'\subseteq k$ with
\[
x,x'\in A',\quad A' \eqpol A,\qquad y,y'\in B',\quad B' \eqpol B,
\]
and a polynomial $Q$ of degree $\gamma $ reducing $A'$ to $B'$ and such that
\[
Q(x)=y,\qquad Q(x')=y'.
\]

Moreover, the multiplicities of $x,x'$ as roots of the polynomials $Q(X)-y,Q(X)-y'$, respectively, are the same as the multiplicities of $a,a'$ as roots of the polynomials $P(X)-b,P(X)-b'$, respectively.
\end{lemma}

\begin{proof}
Let $R$ be a linear polynomial such that $R(x)=a,R(x')=a'$, and let $S$ be a linear polynomial such that $S(b)=y,S(b')=y'$.
Set
\[
A'=R^{-1}(A),\quad B'=S(B),\quad Q=SPR.
\]
\end{proof}

The main use of lemma \ref{lemmaprivileged} is with $x=0,x'=1$, or $y=0,y'=1$: this allows in particular to replace a set with an equivalent set containing $0,1$, which sometimes simplifies calculations.

The following should be contrasted with theorem \ref{030820221312} and proposition \ref{longzchain}.

\begin{proposition} \label{180920220929}
Let $\Theta\in\bigcup_{m\in \N }\Sigma_m$.
Then the set $ \left \{
\Xi\mid\Theta \pol \Xi
\right \} $ is finite.
\end{proposition}

\begin{proof}
By proposition \ref{incrcard} and \eqref{conditiong}, it is enough to show that if $A=\{a_1,\ldots ,a_m\}\in \mathcal P_m$, then for $2\le n<m$ and $ \frac mn \le\gamma\le \frac{m-1}{n-1} $ there are only finitely many classes $[B]$ with $B=\{ 0,1,b_3,\ldots ,b_n\}\in \mathcal P_n$ such that $A=P^{-1}(B)$ for some polynomial $P$ of degree $\gamma $.
If $B$ is such a set and $P$ such a polynomial, then there exist a non-empty set $I=\{ i_1,\ldots ,i_p\}\subseteq\{ 1,\ldots ,m\} $ and an element $j\in\{ 1,\ldots ,m\}\setminus\{ i_1,\ldots ,i_p\} $ such that $P^{-1}(\{ 0\} )=\{ a_{i_1},\ldots ,a_{i_p}\} ,P(a_j)=1$.
Therefore there exist positive integers $r_1,\ldots ,r_p$ with $r_1+\ldots +r_p=\gamma $ such that $P(X)=c(X-a_{i_1})^{r_1}\cdot\ldots\cdot (X-a_{i_p})^{r_p}$ where, from $P(a_j)=c(a_j-a_{i_1})^{r_1}\cdot\ldots\cdot (a_j-a_{i_p})^{r_p}=1$, it follows that $c= \frac 1{(a_j-a_{i_1})^{r_1}\cdot\ldots\cdot (a_j-a_{i_p})^{r_p}} $.
Since $B=P(A)$, this implies that the class $[B]$ depends only on the choice of $i_1,\ldots ,i_p,j,r_1,\ldots ,r_p$; as there are finitely many such choices, the result follows.
\end{proof}

\begin{definition}
Given $n,m,\gamma $ satisfying \eqref{conditiong} and $\Theta\in\Sigma_n$, let $ \mathcal A_m^{\gamma }(\Theta )$ be set of all $\Xi\in\Sigma_m$ such that there exist $B\in\Theta $, $A\in\Xi $, and a polynomial of degree $\gamma $ reducing $A$ to $B$.

Let also
\begin{equation} \label{03112022}
\mathcal A_m(\Theta )=\{\Xi\in\Sigma_m\mid\Xi \pol \Theta\} =\bigcup_{\gamma\in \left [
\frac mn , \frac {m-1}{n-1}
\right ]
\cap \N } \mathcal A_m^{\gamma }(\Theta ).
\end{equation}
\end{definition}

Note that the definition of $ \mathcal A_m^{\gamma }(\Theta )$ does not really depend on the choice of $A,B$, since any sets equivalent to $A,B$, respectively, are bireducible to them by linear polynomials. 

Our goal is to show that $ \mathcal A_m(\Theta )$ is an affine variety and compute its dimension.
Since the union in \eqref{03112022} is finite, this amounts to show this for all $ \mathcal A_m^{\gamma }(\Theta )$.

\begin{theorem} \label{bidimvar}
Fix $n,m$ with $n<m$ and let $\Theta\in\Sigma_n$.
Then $ \mathcal A_m(\Theta )$ is a subvariety of $\Sigma_m$.
\end{theorem}

\begin{proof}
Let $B=\{ 0,1,b_3,\ldots ,b_n\}\in\Theta $.
It is enough to show that every
\begin{equation} \label{290720221556}
\begin{array}{rl}
\mathcal V_{\gamma }= & (\theta'\theta )^{-1}( \mathcal A_m^{\gamma }(\Theta ))= \\
= & \{ (a_1,\ldots ,a_m)\mid\{a_1,\ldots ,a_m\} =P^{-1}( \{ 0,1,b_3,\ldots ,b_n\} ) \text{ for some } \\
 & \text{polynomial } P \text{ of degree } \gamma\} ,
\end{array} \end{equation}
for $ \frac mn \le\gamma\le \frac{m-1}{n-1} $, is a subvariety of $D(f_m)$.

The idea is that we can determine whether a given $(a_1,\ldots ,a_m)$ is in $V_{\gamma }$ by looking at all partitions of $\{ a_1,\ldots ,a_m\} $ in $n$ pieces and trying to write a polynomial of degree $\gamma $ such that the elements of the partition are the sets of preimages of the elements of $B$; in order to express this polynomial, we write its factorisation using one of the elements of the partition as the set of preimages of $0\in B$.

So let $(a_1,\ldots ,a_m)\in D(f_m)$.
Then $(a_1,\ldots ,a_m)\in \mathcal V_{\gamma }$ if and only if there exist:
\begin{itemize}
\item an integer $\ell $ with $1\le\ell\le\gamma $,
\item a set of $\ell $ indices $I_1=\{ i_1,\ldots ,i_{\ell }\}\subseteq\{ 1,\ldots ,m\} $,
\item positive integers $r_{1},\ldots ,r_{{\ell }}$ with $r_{1}+\ldots +r_{{\ell }}=\gamma $,
\item and $c\in k\setminus\{ 0\} $
\end{itemize}
such that letting $P(X)=c\prod_{h=1}^{\ell }(X-a_{i_h})^{r_{h}}$ one has $\{ a_1,\ldots ,a_m\}=P^{-1}(\{ 0,1,b_3,\ldots ,b_n\} )$.

This is in turn equivalent to the following condition:
\begin{itemize}
\item[ ] there exist $\ell $ with $1\le\ell\le\gamma $, a partition $\{ I_1,\ldots ,I_n\} $ of $\{ 1,\ldots ,m\} $ with $ \card (I_1)=\ell $, and positive integers $r_{1},\ldots ,r_{\ell }$ with $\sum_{h=1}^{\ell }r_{h}=\gamma $, such that the following equalities hold:
\begin{equation} \label{varietypol} \left \{ \begin{array}{lll}
c\prod_{h\in I_1}(a_j-a_{h})^{r_h}=1 & \text{for every } j\in I_2 \\
c\prod_{h\in I_1}(a_j-a_{h})^{r_{h}}=b_3 & \text{for every } j\in I_3 \\
\ldots & \ldots \\
c\prod_{h\in I_1}(a_j-a_{h})^{r_h}=b_n & \text{for every } j\in I_n
\end{array} \right . \end{equation}
where $c= \frac 1{\prod_{h\in I_1}(a_j-a_{h})^{r_{h}}} $, for $j$ any fixed element of $I_2$: in fact $c$ could be computed from any of the equations \eqref{varietypol} and plugged into the other equations.
\end{itemize}
Equations \eqref{varietypol} provide polynomial equations in the coordinates $(a_1,\ldots ,a_m)$, therefore they define an affine variety in such coordinates.
This affine variety depends on the choice of the partition $\{ I_1,\ldots ,I_n\} $ and of the integers $r_{1},\ldots ,r_{\ell }$.
Since there are finitely many such choices, $V_{\gamma }$ is a finite union of affine varieties, therefore an affine variety itself.
\end{proof}

To compute the dimension of $ \mathcal A_m(\Theta )$ we start with the following.

\begin{lemma} \label{020820221452}
Let $n,m,\gamma $ satisfy \eqref{conditiong} and fix $B=\{b_1,\ldots ,b_n\}\in \mathcal P_n$.
Let $ \mathcal U $ be the collection of all $(a_1,\ldots ,a_m,c_0,\ldots ,c_{\gamma })\in D(f_m)\times k^{\gamma +1}$ such that $ \{ a_1,\ldots ,a_m\} =P^{-1}(B)$ where $P(X)=\sum_{j=0}^{\gamma }c_jX^j$.
Then $ \mathcal U $ is a subvariety of $D(f_m)\times k^{\gamma +1}$ of dimension $\gamma +1-(n\gamma -m)$.
\end{lemma}

\begin{proof}
Denote $r=n\gamma -m$, so that $0\le r\le\gamma -1$ by \eqref{conditionginv}.

Consider the following condition on a tuple $(a_1,\ldots ,a_m)\in D(f_m)$ and a polynomial $P$ of degree at most $\gamma $:
\begin{itemize}
\item[$(*)$] There exist:
\begin{itemize}
\item a sequence of positive integers $ \vec \ell =(\ell_1,\ldots ,\ell_n)$ such that $\ell_1+\ldots +\ell_n=m$,
\item and a sequence of natural numbers $ \vec s =(s_{1},\ldots ,s_m)$ such that
\begin{equation} \label{12102022}
\left \{ \begin{array}{l}
s_1+s_2+\ldots +s_{\ell_1}+\ell_1=\gamma \\
s_{\ell_1+1}+s_{\ell_1+2}+\ldots +s_{\ell_1+\ell_2}+\ell_2=\gamma \\
\ldots \\
s_{\ell_1+\ldots +\ell_{n-1}+1}+s_{\ell_1+\ldots +\ell_{n-1}+2}+\ldots +s_m+\ell_n=\gamma
\end{array} \right . \end{equation}
(which implies in particular that $s_1+\ldots +s_m=r$)
\end{itemize}
satisfying
\begin{equation} \label{31072022949}
\left \{ \begin{array}{l}
P(a_{1})=\ldots =P(a_{\ell_1})=b_{1} \\
P'(a_{1})=P''(a_{1})=\ldots =P^{(s_{1})}(a_{1})=0 \\
\ldots \\
P'(a_{\ell_1})=\ldots =P^{(s_{\ell_1})}(a_{\ell_1})=0 \\
\hline \\
P(a_{\ell_1+1})=\ldots =P(a_{\ell_1+\ell_2})=b_{2} \\
P'(a_{\ell_1+1})=\ldots =P^{(s_{\ell_1+1})}(a_{\ell_1+1})=0 \\
\ldots \\
P'(a_{\ell_1+\ell_2})=\ldots =P^{(s_{\ell_1+\ell_2})}(a_{\ell_1+\ell_2})=0 \\
\hline \\
\ldots \\
\hline \\
P(a_{\ell_1+\ldots +\ell_{n-1}+1})=P(a_{\ell_1+\ldots +\ell_{n-1}+2})=\ldots =P(a_m)=b_n \\
P'(a_{\ell_1+\ldots +\ell_{n-1}+1})=\ldots =P^{(s_{\ell_1+\ldots +\ell_{n-1}+1})}(a_{\ell_1+\ldots +\ell_{n-1}+1})=0 \\
\ldots \\
P'(a_m)=\ldots =P^{(s_m)}(a_m)=0
\end{array} \right .
.
\end{equation}
\end{itemize}

System \eqref{31072022949} is composed of $n$ blocks, one for each element of $B$.
Notice that $n\ge 2$ implies that $P$ is not a constant polynomial.
Then the first block describes the preimages of $b_1$, saying that $a_{1}$ is a root of the polynomial $P(X)-b_{1}$ of multiplicity at least $s_{1}+1$,\ldots , $a_{\ell_1}$ is a root of the polynomial $P(X)-b_{1}$ of multiplicity at least $s_{\ell_1}+1$; the other blocks have similar meanings.
By \eqref{12102022} this implies that the multiplicity of each $a_j$ as root of the corresponding polynomial $P(X)-b_{i_j}$ is exactly $s_j+1$.
Therefore from $(*)$ and $\deg (P)\le\gamma $ it also follows that $\deg (P)=\gamma $ and $\{ a_1,\ldots ,a_m\} =P^{-1}(B)$.
Conversely, if $a_1,\ldots ,a_n,P$ are such that $\{ a_1,\ldots ,a_n\} =P^{-1}(B)$, for distinct $a_1,\ldots ,a_n$, then there exists $\sigma\in Sym_n$ such that $(a_{\sigma (1)},\ldots ,a_{\sigma (n)},P)$ satisfies \eqref{31072022949}.

Equations \eqref{31072022949} define a subvariety $ \mathcal U_{ \vec \ell \vec s }$ ---depending on the choice of the sequences $ \vec \ell , \vec s $--- of $D(f_m)\times k^{\gamma +1}$ in the coordinates $(a_1,\ldots ,a_m,c_0,\ldots ,c_{\gamma })$, where $c_j$ is the coefficient of degree $j$ of $P$.
Then $ \mathcal U $ is the union of the images under all permutations of the variables $(a_1,\ldots ,a_n)$ of such varieties $ \mathcal U_{ \vec \ell \vec s }$.
Since such a union is finite, $ \mathcal U $ is a subvariety of $D(f_m)\times k^{\gamma +1}$ whose dimension is the maximum of the dimensions of all $ \mathcal U_{ \vec \ell \vec s }$.
Therefore it is enough to establish the following two facts:
\begin{itemize}
\item[(a)] some $ \mathcal U_{ \vec \ell \vec s }$ is non-empty;
\item[(b)] if $ \mathcal U_{ \vec \ell \vec s }$ is non-empty, then  $\dim ( \mathcal U_{ \vec \ell \vec s })=\gamma +1-r$.
\end{itemize}

For part (a) it is enough to show that there exists at least a polynomial $P$ of degree $\gamma $ such that $ \card (P^{-1}(B))=m$.
Pick $b\in B$ and consider the polynomials $P_{\mu }(X)=\mu X^{r+1}(X-1)(X-2)\cdot\ldots\cdot (X-(\gamma -r-1))+b$, for $\mu\in k\setminus\{ 0\} $.
Notice that the set $S$ of roots of the derivative $P'_{\mu }$ does not depend on $\mu $, and $0\in S$ if and only if $r>0$.
Also, $ \card (P_{\mu }^{-1}(B))\le m$, while $ \card (P_{\mu }^{-1}(B))=m$ if and only if $\forall a\in S\setminus\{ 0\} ,P_{\mu }(a)\notin B$: this last fact holds for all but finitely many values of $\mu $.

We establish (b) by computing the Jacobian matrix of \eqref{31072022949} and showing that it is non-singular at every point of $ \mathcal U_{ \vec \ell \vec s }$: since this Jacobian is a $(m+r)\times (m+\gamma +1)$ matrix, it follows that indeed $\dim ( \mathcal U_{ \vec \ell \vec s })=m+\gamma +1-(m+r)=\gamma +1-r$.

System \eqref{31072022949} means that
\begin{equation} \label{07092022}
\left \{ \begin{array}{lcl}
c_0+c_1a_{j}+c_2a_{j}^2+c_3a_{j}^3+\ldots +c_{\gamma -1}a_{j}^{\gamma -1}+c_{\gamma }a_{j}^{\gamma } & = & b_{i_j} \\
c_1+2c_2a_{j}+3c_3a_{j}^2+\ldots +(\gamma -1)c_{\gamma -1}a_{j}^{\gamma -2}+\gamma c_{\gamma }a_{j}^{\gamma -1} & = & 0 \\
2c_2+6c_3a_{j}+\ldots +(\gamma -2)(\gamma -1)c_{\gamma -1}a_{j}^{\gamma -3}+(\gamma -1)\gamma c_{\gamma }a_{j}^{\gamma -2} & = & 0 \\
\ldots & & \\
s_{j}!c_{s_{j}}+\ldots +(\gamma -s_{j})\cdot\ldots\cdot (\gamma -1)c_{\gamma -1}a_{j}^{\gamma -s_{j}-1}+(\gamma -s_{j}+1)\cdot\ldots\cdot\gamma c_{\gamma }a_{j}^{\gamma -s_{j}} & = & 0
\end{array} \right . \end{equation}
for $j\in\{ 1,\ldots ,m\} $, where $b_{i_j}$ is the image of $a_j$ under $P$ in \eqref{31072022949}.
The part of the Jacobian of \eqref{31072022949} corresponding to this block of equations is a $(s_j+1)\times(m+\gamma +1)$ submatrix; the first $m$ columns, containing the partial derivatives with respect to the variables $(a_1,\ldots ,a_m)$ are
\begin{equation} \label{04092022}
\left ( \begin{matrix}
0 & \ldots & 0 & P'(a_j) & 0 & \ldots & 0 \\
0 & \ldots & 0 & P''(a_j) & 0 & \ldots & 0 \\
\ldots & \ldots & \ldots & \ldots & \ldots & \ldots & \ldots \\
0 & \ldots & 0 & P^{(s_j+1)}(a_j) & 0 & \ldots & 0
\end{matrix} \right ) \end{equation}
while the last $\gamma +1$ columns, containing the partial derivatives with respect to the variables $(c_0,\ldots ,c_{\gamma })$ are
\begin{equation} \label{04082022}
\left ( \begin{matrix}
1 & a_{j} & a_{j}^2 & a_{j}^3 & \ldots & a_{j}^{\gamma -1} & a_{j}^{\gamma } \\
0 & 1 & 2a_{j} & 3a_{j}^2 & \ldots & (\gamma -1)a_{j}^{\gamma -2} & \gamma a_{j}^{\gamma -1} \\
0 & 0 & 2 & 6a_{j} & \ldots & (\gamma -2)(\gamma -1)a_{j}^{\gamma -3} & (\gamma -1)\gamma a_{j}^{\gamma -2} \\
\ldots & \ldots & \ldots & \ldots & \ldots & \ldots & \ldots \\
0 & 0 & 0 & 0 & \ldots & (\gamma -s_{j})\cdot\ldots\cdot (\gamma -1)a_{j}^{\gamma -s_{j}-1} & (\gamma -s_{j}+1)\cdot\ldots\cdot\gamma a_{j}^{\gamma -s_{j}}
\end{matrix} \right )
.
\end{equation}
Note that each of the submatrices \eqref{04092022} has exactly one non-null entry when computed in a point of $ \mathcal U_{ \vec \ell \vec s }$, namely
\begin{equation} \label{040920221008}
P^{(s_j+1)}(a_j).
\end{equation}
Let $M$ be the $r\times (\gamma +1)$ matrix obtained from the Jacobian by erasing the rows and columns corresponding to the entries \eqref{040920221008}.
Then the proof is completed by the following claim, which will be established in section \ref{310720221554}.

\begin{claim}
Matrix $M$ is non-singular.
\end{claim}
\end{proof}

\begin{theorem} \label{030820221312}
Let $n,m$ satify \eqref{conditiong} and let $\Theta\in\Sigma_n$.
Let $m=en+t$, where $e$ is the integer part of $ \frac mn $.
Then
\[
\dim ( \mathcal A_m(\Theta ))= \begin{cases}
e-1 & \text{if } t=0 \\
e-n+t & \text{if } t>0
\end{cases}
.
\]
\end{theorem}

Notice that in the second case indeed $e-n+t\ge 0$ by \eqref{e}, since if $m\in [\gamma (n-1)+1,\gamma n]$, then $e<\gamma $, so $e-n+t=e-n+m-en=m-n-(n-1)e\ge\gamma (n-1)+1-n-(n-1)(\gamma -1)=0$.

\begin{proof}[Proof of theorem \ref{030820221312}]
Let $B=\{ b_1,\ldots ,b_n\}\in\Theta $.
Let $ \mathcal A'(\gamma )$ be the set of all $(a_1,\ldots ,a_m)\in D(f_m)$ such that $\{ a_1,\ldots ,a_m\} =P^{-1}(B)$ for some polynomial $P$ of degree $\gamma $.
By lemma \ref{020820221452}, $ \mathcal U = \mathcal U (\gamma )=\{ (a_1,\ldots ,a_m,c_0,\ldots ,c_{\gamma })\in D(f_m)\times k^{\gamma +1}\mid\{ a_1,\ldots ,a_m\} =P^{-1}(B), \text{ where } P=\sum_{j=1}^{\gamma }c_jX^j\} $ is a subvariety of $D(f_m)\times k^{\gamma +1}$ whose projection onto the first $m$ coordinates is $ \mathcal A'(\gamma )$.
By lemma \ref{040920221508}, for every $(a_1,\ldots ,a_m)\in \mathcal A'(\gamma )$ there are only finitely many $(c_0,\ldots ,c_{\gamma })\in k^{\gamma +1}$ such that $(a _1,\ldots ,a_m,c_0,\ldots c_{\gamma })\in \mathcal U (\gamma )$, which implies that $\dim ( \mathcal A'(\gamma ))=\gamma +1-(n\gamma -m)$ by lemma \ref{020820221452}.
Since $\dim ( \mathcal A_m(\Theta ))=\max\{\dim ( \mathcal A_m^{\gamma }(\Theta ))\}_{\gamma\in \left [
\frac mn , \frac{m-1}{n-1}
\right ]
\cap \N }$ and $\dim ( \mathcal A_m^{\gamma }(\Theta ))=\dim ( \mathcal A'(\gamma ))-2$, it is enough to show that
\[
\max\{\dim ( \mathcal A'(\gamma ))\}_{\gamma\in \left [
\frac mn , \frac{m-1}{n-1}
\right ]
\cap \N } = \begin{cases}
e+1 & \text{if } t=0 \\
e-n+t+2 & \text{if } t>0
\end{cases}
.
\]

{\sl Case 1.} $t=0$.

The biggest value of $\gamma +1-(n\gamma -m)$, in the allowed range \eqref{conditiong}, is obtained for $\gamma =e$, yielding $\dim ( \mathcal A'(e))=e+1$.

{\sl Case 2.} $t>0$.

Since \eqref{conditiong} implies $ \frac{ne+t}n =e+ \frac tn \le\gamma $, the biggest value of $\gamma +1-(n\gamma -m)$ is obtained for $\gamma =e+1$, yielding $\dim ( \mathcal A'(e+1))=e+2-(n(e+1)-m)=e-n+t+2$.
\end{proof}

\begin{theorem} \label{twontwo}
Let $B\in \mathcal P_n$.
\begin{enumerate}
\item If there exists a quadratic polynomial reducing $A$ to $B$, then $A$ is an exceptional set and either $[A]\in \mathcal E_{2n-1}^*$ or $[A]\in \mathcal E_{2n}^*$, that is, the group $G_A$ has even order.
\item There is a unique $\Theta\in\Sigma_{2n-1}$ such that $\Theta \pol [B]$.
Conversely, if $\Theta\in \mathcal E_{2n-1}^*$, then there exists a unique $\Xi\in\Sigma_n$ such that $\Theta \pol \Xi $.
\item Let $ \mathcal A $ be the family of all $[A]\in\Sigma_{2n}$ such that there exists a quadratic polynomial reducing $A$ to $B$.
Then $\dim ( \mathcal A )=1$.
Conversely, if $\Theta\in \mathcal E_{2n}^*$, then there exists $\Xi\in\Sigma_n$ and a quadratic polynomial $P$ reducing some element of $\Theta $ to some element of $\Xi $.
\end{enumerate}
\end{theorem}

Note that in case (3) of theorem \ref{twontwo}, if $n>2$ the condition that the polynomials be quadratic is redundant, since it is implied by \eqref{conditiong}; so $ \mathcal A = \mathcal A_{2n}([B])$.
For $n=2$, a polynomial reducing an element of $ \mathcal P_{2n}$ to an element of $ \mathcal P_n$ may have degree $3$.

\begin{proof}[Proof of theorem \ref{twontwo}]
Let $P$ be a quadratic polynomial reducing $A$ to $B$.
We distinguish two cases, according to whether $A\in \mathcal P_{2n-1}$ or $A\in \mathcal P_{2n}$.

Let first $A\in \mathcal P_{2n-1}$.
In this case, there is a unique element in $B$ having exactly one preimage under $P$, while all other elements of $B$ have two preimages.
By applying lemma \ref{lemmaprivileged} it can be assumed that $B=\{ 0,1,b_3,\ldots ,b_n\} $, that $\{ 0,1\}\subseteq A$, that $0$ is a root of $P$ of multiplicity $2$, and that $P(1)=1$.
It follows that $P(X)=X^2$, so that $A=\{ 0,\pm 1,\pm a_3,\ldots ,\pm a_n\} $ where $a_j$ is a square root of $b_j$, establishing in particular the existence and uniqueness of $[A]$, that is the first part of (2).
Now apply corollary \ref{witntwo} to $\tau :A\to A$ defined by $\tau (a )=-a $ to conclude that $A$ is an exceptional set.
The fact that $A$ has a characteristic permutation of order $2$ proves that $ \card (G_A)$ is even.
Therefore (1) is established in this case.

We now apply this argument to show that the correspondence $[B]\mapsto [A]$ is injective, which gives unicity in the second part of (2).
Given $B=\{ 0,1,b_3,\ldots ,b_n\} ,B'=\{ 0,1,b'_3,\ldots ,b'_n\} $ and $A=\{ 0,\pm 1,\pm a_3,\ldots ,\pm a_n\} ,A'=\{ 0,\pm 1,\pm a'_3,\ldots ,\pm a'_n\} $ such that $A=P^{-1}(B),A'=P^{-1}(B')$ where $P(X)=X^2$, assume that $A \eqpol A'$.
Then let $Q$ be a linear polynomial such that $Q(A)=A'$.
Since $Q$ preserves barycentres, it follows that $Q(0)=0$, so that $Q(X)=cX$ for some $c\in k\setminus\{ 0\} $, whence $PQ(X)=c^2X^2$.
Therefore $B'=PQ(A)=\{ 0,c^2,c^2a_3^2,\ldots ,c^2a_n^2\} =R(B) \eqpol B$, where $R(X)=c^2X$.

Now let $A\in \mathcal P_{2n}$.
Then every element of $B$ has exactly two preimages under $P$.
By lemma \ref{lemmaprivileged} again, it can be assumed that:
\begin{itemize}
\item $A=\{ 0,1,a_3,\ldots ,a_{2n}\} $,
\item $B=\{ 0,1,b_3,\ldots ,b_n\} $,
\item $P(0)=P(1)=0$, $P(a_3)=P(a_4)=1$, and $P(a_{2j-1})=P(a_{2j})=b_j$ for $3\le j\le n$.
\end{itemize}
It follows that for some $c\in k\setminus\{ 0\} $:
\[ \left \{ \begin{array}{ll}
P(X)=cX(X-1) & \\
P(X)-1=c(X-a_3)(X-a_4) & \\
P(X)-b_j=c(X-a_{2j-1})(X-a_{2j}) & \text{for } 3\le j\le n
\end{array} \right .
.
\]
Then, from the equalities
\[
cX(X-1)=c(X-a_3)(X-a_4)+1=c(X-a_{2j-1})(X-a_{2j})+b_j,
\]
that is
\begin{equation} \label{eqpolys}
\begin{aligned}
cX^2-cX= & cX^2-c(a_3+a_4)X+ca_3a_4+1= \\
= & cX^2-c(a_{2j-1}+a_{2j})X+ca_{2j-1}a_{2j}+b_j,
\end{aligned}
\end{equation}
it follows that
\[
\forall j\in\{ 2,\ldots ,n\} ,a_{2j-1}+a_{2j}=1.
\]
Therefore, let $\tau :A\to A$ be defined by
\[
\tau (0)=1,\quad\tau (1)=0,\quad\tau (a_{2j-1})=a_{2j},\quad\tau (a_{2j})=a_{2j-1}
\]
for $j\in\{ 2,\ldots ,n\} $, and apply corollary \ref{witntwo} to obtain that $A$ is an exceptional set.
The fact that $\tau $ has order $2$ shows again that $ \card (G_A)$ is even.
This concludes the proof of (1).

The first part of (3) is a consequence of the proof of lemma \ref{020820221452}: when $\gamma =2,r=0$ it follows that $\dim ( \mathcal U_{ \vec \ell \vec s })=3$ for non-empty $ \mathcal U_{ \vec \ell \vec s }$, whence $\dim ( \mathcal A )=3-2=1$.

It remains to prove the existence in the second part of (2) and the second part of (3).
Let $A$ be an exceptional set such that $ \card (G_A)$ is even, say $ \card (G_A)=r$, and let $\tau $ be a generator of $G_A$, so that $\tau $ has order $r$.
Therefore, $\rho =\tau^{ \frac r2 }$ has order $2$.

Suppose that $A\in \mathcal P_{2n-1}$.
Using lemma \ref{lemmaprivileged}, it can be assumed that $0$ is the fixed point of $\rho $.
So $\forall a\in A,a+\rho (a)=0$ by corollary \ref{witntwo}, and it follows that $A=\{ 0,\pm a_2,\ldots ,\pm a_n\} $ for some $a_2,\ldots ,a_n$.
This implies that $A=P^{-1}(\{ 0,a_2^2,\ldots ,a_n^2\} )$ where $P(X)=X^2$.

Assume now that $A\in \mathcal P_{2n}$.
By lemma \ref{lemmaprivileged} it can be assumed that $\pm 1\in A$, with $\rho (\pm 1)=\mp 1$.
Since $\forall a\in A,a+\rho (a)=0$ by corollary \ref{witntwo}, it follows that $A=\{\pm 1,\pm a_2,\ldots ,\pm a_n\} $ for some $a_2,\ldots ,a_n$.
This implies that $A=P^{-1}(\{ 1,a_2^2,\ldots ,a_n^2\} )$ where $P(X)=X^2$.
\end{proof}

Theorem \ref{twontwo}(2) shows that using $ \pol $ the points of $\Sigma_n$ can be parameterised by the elements of $ \mathcal E_{2n-1}^*$.
The following proposition shows that no two elements of $\Sigma_n$ have a common predecessor in $\Sigma_{2n}$.
Therefore, by theorem \ref{twontwo}(3), the relation $ \pol $ determines a partition of $ \mathcal E_{2n}^*$ into one-dimensional subvarieties, one for each element of $\Sigma_n$.

\begin{proposition} \label{nincompatible}
Let $\Xi ,\Xi'\in\Sigma_n$ and $\Theta\in\Sigma_{2n}$ be such that both $\Theta \pol \Xi $ and $\Theta \pol \Xi'$ both hold.
Then $\Xi =\Xi'$.
\end{proposition}

\begin{proof}
It can be assumed that $n>2$.
If $A\in\Theta ,B\in\Xi $, any polynomial $P$ reducing $A$ to $B$ is quadratic.
By lemma \ref{lemmaprivileged} it can be assumed that $A=\{ 0,1,a_3,\ldots ,a_{2n}\} ,B=\{ 0,1,b_3,\ldots ,b_n\} ,P(0)=P(1)=0$.
Therefore, $P(X)=cX(X-1)$ for some $c\in k\setminus\{ 0\} $.
Again by lemma \ref{lemmaprivileged}, let $B'=\{ 0,1,b'_3,\ldots ,b'_n\}\in\Xi'$ and $Q$ be a quadratic polynomial such that $A=Q^{-1}(B'),Q(0)=Q(1)=0$.
It follows that $Q(X)=c'X(X-1)$ for some $c'\in k\setminus\{ 0\} $.
Thus $B=R^{-1}(B')$, where $R(X)= \frac {c'}c X$, so $B \eqpol B'$.
\end{proof}

Theorem \ref{twontwo} and proposition \ref{nincompatible} say that two $ \eqpol $-incomparable sets with $n$ elements do not have a common lower bound with $2n-1$ or $2n$ elements.
This suggests the following.

\begin{question} \label{05092022}
Given non-empty $B,B'\in Fin$, when do they have a common lower bound other than $\emptyset ,k$? That is, when do there exist $A\in Fin\setminus\{\emptyset ,k\} $ such that $A \pol B$ and $A \pol B'$?
\end{question}

Examples of pairs of $ \eqpol $-incomparable sets having a common lower bound different from $\emptyset ,k$ can be constructed as follows.
Let $P(X),Q(X)$ be commuting polynomials, that is such that $P[Q(X)]=Q[P(X)]$.
Fix any $C\in Fin$ and let $B=P^{-1}(C),B'=Q^{-1}(C)$.
Then $A=Q^{-1}(B)=P^{-1}(B')$ is a common lower bound of $B,B'$.
There are indeed choices of $C,P(X),Q(X)$ so that $B,B'$ are $ \pol $-incomparable.
For example, for $j\ne j'$, let $P(X)=X^j,Q(X)=X^{j'}$, or let $P(X),Q(X)$ be the $j$-th and the $j'$-th Chebyshev polynomials, respectively: if $ \card (P^{-1}(C)), \card (Q^{-1}(C))$ do not satisfy \eqref{between}, then $B,B'$ are $ \pol $-incomparable.
For instance, starting with the one-element set $\{ 1\} $ and closing under the preimages of all powers $X^j$, one obtains the sublattice $ \mathcal R $ consisting of the sets of $j$-th roots of unity, for any $j$; notice that these are exceptional sets by theorem \ref{chrexcsets}.
On $ \mathcal R $, the relation $ \pol $ coincide with reverse inclusion $\supseteq $.

However we do not know how is the general situation.

\subsubsection{Large antichains and maximal elements}
In this section we discuss the existence of large antichains with respect to $ \pol $ among finite sets.
For sets that are infinite and coinfinite, see theorem \ref{maximalinfinite}.

As a first remark, notice that by corollary \ref{reductioniffequivalence} every $\Sigma_m$ is an antichain, which has cardinality $\kappa $ by \cite[exercise I.4.8]{hartsh1977} if $m\ge 3$.
Moreover, using an argument as in the proof of lemma \ref{020820221452}, we can find antichains that extend across all $\Sigma_m$ and are big, namely contain an open set, in each $\Sigma_m$.
In fact, while there are no $ \pol $-minimal polynomial classes of finite sets but $[\emptyset ]$ (see theorem \ref{030820221312}), our argument proves that most classes are $ \pol $-maximal.

\begin{theorem} \label{propactwo}
There exists an antichain $ \mathcal C $ in $\bigcup_{m\ge 2}\Sigma_m$ such that $ \mathcal C \cap\Sigma_m$ contains an open set of $\Sigma_m$ for every $m$.
In particular, for $m\ge 3$ every $ \mathcal C \cap\Sigma_m$ has cardinality $\kappa $.
\end{theorem}

\begin{proof}
The antichain $ \mathcal C $ consists of all elements in every $\Sigma_m$ that are maximal, below the unique element of $\Sigma_1$, with respect to $ \pol $.
It is therefore enough to prove that for every $m\ge 3$ the elements of $\Sigma_m$ that are not $ \pol $-maximal below the $1$-element sets are contained in a proper subvariety of $\Sigma_m$; equivalently, that the set
\begin{multline*}
\mathcal S =\{ (a_1,\ldots ,a_m)\in D(f_m)\mid \{ a_1,\ldots ,a_m\} \text{ is polynomially reducible to a } \\
n \text{-element set for some } n\ge 2\}
\end{multline*}
is contained in a proper subvariety of $D(f_m)$.
The tuple $(a_1,\ldots ,a_m)$ is in $ \mathcal S $ if and only if there exist $n\ge 2$ and $\gamma $ satisfying \eqref{conditiong} such that $\{ a_1,\ldots ,a_m\} $ is reducible to some member of $ \mathcal P_n$ by a polynomial of degree $\gamma $.

Consider again the system constituted by the blocks \eqref{07092022}, for $j\in\{ 1,\ldots ,m\} $, but now in the variables $a_1,\ldots ,a_m,b_1,\ldots ,b_n,c_0,\ldots ,c_{\gamma +1}$.
This defines a subvariety $ \mathcal W_{ \vec \ell \vec s \vec i }$ of $D(f_m)\times D(f_n)\times k^{\gamma +1}$, which depends on the choice of the sequences $ \vec \ell , \vec s $ as in condition $(*)$ of the proof of lemma \ref{020820221452}, but also of the sequence $ \vec i =(i_1,\ldots ,i_m)$ such that $\{ 1,\ldots ,n\} =\{ i_1,\ldots ,i_m\} $.
Then $ \mathcal S $ is the union of the projections on $D(f_m)$ of the images of the varieties $ \mathcal W_{ \vec \ell \vec s \vec i }$ under any permutations of $a_1,\ldots ,a_m$.

Now we determine $\dim ( \mathcal W_{ \vec \ell \vec s \vec i })$, again by computing the Jacobian.
This is as the Jacobian discussed in the proof of lemma \ref{020820221452}, with $n\gamma $ rows, except that every row is longer, since it contains also the derivatives with respect to $b_1,\ldots ,b_n$.
In particular, it is non-singular at every point of $ \mathcal W_{ \vec \ell \vec s \vec i }$.
Therefore $\dim ( \mathcal W_{ \vec \ell \vec s \vec i })=m+n+\gamma +1-n\gamma <m$ if and only if
\begin{equation} \label{23102022}
(n-1)\gamma >n+1.
\end{equation}
In this case, the projection of $ \mathcal W_{ \vec \ell \vec s \vec i }$ on $D(f_m)$ has also dimension less than $m$.

It remains to deal separately with the cases when \eqref{23102022} fails, that is:
\[
n=2,\gamma =2,\quad n=3,\gamma =2,\quad n=2,\gamma =3.
\]
When $n=2$ use theorem \ref{030820221312} recalling that $\Sigma_2$ contains just one element; when $\gamma =2$ use theorem \ref{twontwo}(1) and corollary \ref{170920221716}.
\end{proof}

\section{Sets that are infinite and coinfinite} \label{180920220956}
We discuss now subsets of $k$ that are infinite and coinfinite.
Recall indeed from proposition \ref{cardinality} that such sets are $ \pol $-incomparable with sets that are finite or cofinite; moreover if the infinite sets $A,B$ are $ \pol $-comparable, then
\begin{equation} \label{tornacardinalita}
\card (A)= \card (B),\qquad \card (k\setminus A)= \card (k\setminus B).
\end{equation}

By the remark after proposition \ref{orderedfields}, the class of singletons and the class of their complements are maximal elements with respect to $ \pol $.
There are as many other maximal classes as there can be.

\begin{theorem} \label{maximalinfinite}
There exist $2^{\kappa }$ maximal elements in $\Sigma_{\kappa }\cap \check{\Sigma }_{\kappa }$.
\end{theorem}

\begin{proof}
Recall that if $A \lpol B$ then there exists a polynomial $P$ of degree at least $2$ reducing $A$ to $B$.
Using remark \ref{04042022732}(2), it is enough to build $2^{\kappa }$ inequivalent sets $A\in \mathcal P_{\kappa }\cap \check{ \mathcal P }_{\kappa }$, such that, for every polynomial $P$ of degree at least $2$, $A\ne P^{-1}(P(A))$.

Let $\{ P_{\alpha }\}_{\alpha <\kappa }$ be an enumeration of all polynomials of degree at least $2$.
By recursion, let $\alpha <\kappa $ and assume that elements $a_{\beta },b_{\beta },c_{\beta }\in k$, for every $\beta <\alpha $, have been defined such that all elements $a_{\beta },b_{\beta },c_{\beta }$ for $\beta <\alpha $ are distinct and with the property that $\forall\beta <\alpha\ P_{\beta }(a_{\beta })=P_{\beta }(b_{\beta })$.

Given any $d\in k$, the multiple roots of the polynomial $P_{\alpha }(X)-d$ are also roots of $P_{\alpha }'(X)$.
Since $ \card (k\setminus P_{\alpha }(\{ a_{\beta },b_{\beta },c_{\beta }\}_{\beta <\alpha }))=\kappa $, there must exist $d\in k\setminus P_{\alpha }(\{ a_{\beta },b_{\beta },c_{\beta }\}_{\beta <\alpha })$ such that $P_{\alpha }(X)-d$ has no multiple roots, in particular $P_{\alpha }^{-1}(\{ d\} )$ contains at least two elements: call them $a_{\alpha },b_{\alpha }$, fix also any $c_{\alpha }\in k\setminus (\{ a_{\beta },b_{\beta },c_{\beta }\}_{\beta <\alpha }\cup\{ a_{\alpha },b_{\alpha }\} )$.

Let now $A$ be any set containing $\{ a_{\alpha }\}_{\alpha <\kappa }$ and disjoint from $\{ b_{\alpha }\}_{\alpha <\kappa }$.
By construction, $\forall\alpha<\kappa ,b_{\alpha }\in P_{\alpha }^{-1}(P_{\alpha }(A))\setminus A$, so that $A$ satisfies the requirement.

Finally, there are $2^{\kappa }$ such sets that are pairwise inequivalent.
Indeed, there are $2^{\kappa }$ sets of the form $\{ a_{\alpha }\}_{\alpha <\kappa }\cup C$, where $C\subseteq \{ c_{\alpha }\}_{\alpha <\kappa }$; since every polynomial class has at most $\kappa $ elements, there is a subfamily of cardinality $2^{\kappa }$ of inequivalent such sets.
\end{proof}

While theorem \ref{maximalinfinite} states that many subsets of $k$ are maximal with respect to $ \pol $, the argument of the proof can be used to show in fact that in the countable case most subsets of $k$ are maximal, from a topological viewpoint.
This can be viewed as an analogue for infinite sets of the fact that, geometrically, most polynomial classes of finite sets are $ \pol $-maximal below the unique element of $\Sigma_1$, as established in the proof of theorem \ref{propactwo}.

Endow $ \mathcal P (k)$ with the topology whose basis is given by the sets
\[
U_{u_1\ldots u_n}^{v_1\ldots v_m}=\{ A\in \mathcal P (k)\mid u_1,\ldots ,u_n\in A,v_1\ldots v_m\notin A\}
\]
for any choice of distinct $u_1,\ldots ,u_n,v_1,\ldots ,v_m\in k$.
When $k$ is countable, this topology is homeomorphic to the Cantor space.

\begin{theorem} \label{050420221932}
Assume that $k$ is countable.
Then the family of $ \pol $-maximal subsets of $k$ is a comeagre subset of $ \mathcal P (k)$.
\end{theorem}

\begin{proof}
Let $A\in \mathcal P (k)$.
Then $A$ is $ \pol $-maximal if and only if for every polynomial $P$ of degree at least $2$:
\begin{itemize}
\item[(a)] either $P$ is not a reduction of $A$ to $P(A)$
\item[(b)] or there exists a polynomial $Q$ such that $P(A)=Q^{-1}(A)$
\end{itemize}
We show that the family of $ \pol $-maximal sets contains a dense $G_{\delta }$ set.
Since there are countably many polynomials, it is then enough to show that, for any given $P$, the family of sets $A$ satisfying (a) is an open dense subset of $ \mathcal P (k) $.

Condition (a) can be written as $A\ne P^{-1}[P(A)]$; since $A\subseteq P^{-1}[P(A)]$ always holds, in turn (a) is equivalent to
\begin{equation} \label{040420221719}
\exists a\in k,(P(a)\in P(A)\land a\notin A)
\end{equation}
For fixed $a,P$, both conditions $P(a)\in P(A)$ and $a\notin A$ are clopen in the argument $A$: the latter by definition of the topology of $ \mathcal P (A)$; the former because it is equivalent to $a_1\in A\lor\ldots\lor a_l\in A$, where $\{ a_1,\ldots ,a_l\} =P^{-1}[P(\{ a\})]$, and again by the definition of the topology.
Therefore, for any fixed $P$, condition \eqref{040420221719} defines an open set.

To show density, fix distinct elements $u_1,\ldots ,u_n,v_1,\ldots ,v_m\in k$, in order to find some $A\in U_{u_1\ldots u_n}^{v_1\ldots v_m}$ satisfying \eqref{040420221719}.
Let $b\in k\setminus\{ P(u_1),\ldots ,P(u_n),P(v_1),\ldots ,P(v_m)\} $ such that $P^{-1}(\{ b\} )$ contains at least two distinct elements, and let $u,v$ be two of them.
Then any set $A$ containing $\{ u,u_1,\ldots ,u_n\} $ and disjoint from $\{ v,v_1,\ldots ,v_m\} $, that is any element of $U_{uu_1\ldots u_n}^{vv_1\ldots v_m}$, does.
\end{proof}

\begin{question} \label{050420221933}
Do there exist $ \pol $-minimal polynomial classes above $[\emptyset ],[k]$?
\end{question}

The same ideas of the proof of theorem \ref{050420221932} show that, when $k$ is countable, even if question \ref{050420221933} had a positive answer, from the topological point of view the family of $ \pol $-minimal subsets of $k$ above $\emptyset ,k$ is very small.

\begin{proposition}
Assume that $k$ is countable.
Then the family of subsets of $k$ that are $ \pol $-minimal above $\emptyset ,k$ is a meagre subset of $ \mathcal P (k)$.
\end{proposition}

\begin{proof}
Let $A\in \mathcal P (k)$.
Then $A$ is minimal above $\emptyset ,k$ if and only if for every polynomial $P$ of degree at least $2$ there exists a polynomial $Q$ such that $A=Q^{-1}[P^{-1}(A)]$.
In other words, such a subset of $ \mathcal P (k)$ can be written as $\bigcap_P\bigcup_Q \mathcal A_{PQ}$, where $ \mathcal A_{PQ}=\{ A\in \mathcal P (k)\mid A=Q^{-1}[P^{-1}(A)]\} $.
So it is enough to show that every $ \mathcal A_{PQ}$ is a closed set with empty interior.

The fact that each $ \mathcal A_{PQ}$ is closed holds since $A\in \mathcal A_{PQ}$ if and only if
\begin{equation} \label{080420221014}
\forall a\in A, a\in A\Leftrightarrow PQ(a)\in A
\end{equation}

To show that the complement is dense, fix distinct $u_1,\ldots ,u_n,v_1,\ldots ,v_m\in k$ in order to show that there exists $A\in U_{u_1\ldots u_n}^{v_1\ldots v_m}$ such that $A$ does not satisfy \eqref{080420221014}.
As the polynomial $PQ$ has degree at least $2$, pick any distinct $u,v\in k\setminus\{ u_1,\ldots ,u_n,v_1,\ldots ,v_m\} $ such that $v\notin PQ(\{ v_1,\ldots ,v_m\} )$ and $PQ(u)=v$.
Then any set $A$ containing $\{ u,u_1,\ldots ,u_n\} $ and disjoint from $\{ v,v_1,\ldots ,v_m\} $, that is any element of $U_{uu_1\ldots u_n}^{vv_1\ldots v_m}$, does.
\end{proof}

\subsection{Subfields of $ \C $}
In this section we add the extra assumption $k\subseteq \C $, so that we can exploit some properties of the Euclidean topology.

Note that, in addition to satisfy \eqref{tornacardinalita}, if $A,B\notin\{\emptyset ,k\} $ are $ \pol $-comparable then if one of them is bounded, the other one is bounded as well; if the complement of one of them is bounded, then the complement of the other one is bounded as well.
The reason is that the preimage under a non-constant polynomial function of a bounded (respectively, unbounded) subset of $ \C $ is bounded (respectively, unbounded).

\begin{proposition} \label{longzchain}
There exists a $ \pol $-chain of unbounded and counbounded subsets of $k$ that has order type $\zeta $.
\end{proposition}

\begin{proof}
For $n\in \N ,m\in \Z $ denote $ \mathcal C_n^m=\{ x\in \C \mid n^{2^m}\le |x|\le (n+1)^{2^m}\} $.
For $m\in \Z $ let
\[
C_m=\bigcup_{n\in \N } \mathcal C_{2n}^m,\qquad A_m=C_m\cap k,
\]
thus, if $P(X)=X^2$ then $\forall m\in \Z ,A_m=P^{-1}(A_{m+1})$.
Therefore it is enough to show that $\forall m\in \Z ,A_{m+1} \npol A_m$.

Toward contradiction, suppose that $Q(X)=cX^{\gamma }+\ldots\in k[X]$ is a polynomial of degree $\gamma $ such that $A_{m+1}=Q^{-1}(A_m)$, from which it follows also that $Q(A_{m+1})=A_m$.
Note that $\gamma\ge 2$.

Denote by $ \hat Q : \C \to \C $ the polynomial function extending $Q$ to the complex field, which we endow with the Euclidean topology.
So $ \hat Q $ is surjective and continuous.
By the density of $k$ in $ \C $ it follows that $ \hat Q (C_{m+1})=C_m$.
Since every $ \mathcal C_h^{m+1}$ is connected, it follows that for every $h$ there exists a unique $l$ such that
\begin{equation} \label{htol}
\hat Q ( \mathcal C_h^{m+1})\subseteq \mathcal C_l^m.
\end{equation}
If $\varphi : \N \to \N $ is the function assigning to each $h$ the unique $l$ as is \eqref{htol}, by the continuity if $ \hat Q $ with respect to the Euclidean topology we get that $\forall h\in \N ,\varphi (h+1)=\varphi (h)\pm 1$.
Since there must exist $h$ such that $\varphi (h)=0$, it follows that eventually $\varphi (h)\le h$.
In particular, eventually
\begin{equation} \label{zchains}
| \hat Q ((h+1)^{2^{m+1}})|\le (h+1)^{2^m}.
\end{equation}
However
\[
\forall\varepsilon\in \R^+,\exists M\in \R^+,\forall x\in \C ,|x|>M\Rightarrow \left |
\frac{| \hat Q (x)|}{|x|^{\gamma }} -|c|
\right |
<\varepsilon
\]
where the last inequality can be rewritten as $(|c|-\varepsilon )|x|^{\gamma }<| \hat Q (x)|<(|c|+\varepsilon )|x|^{\gamma }$.
Therefore, for $h$ big enough, $| \hat Q ((h+1)^{2^{m+1}})|> \frac{|c|}2 (h+1)^{2^{m+1}\gamma }$, contradicting \eqref{zchains}.
\end{proof}

\begin{remark} \label{060520221239}
The fact that for finite sets polynomial bireducibility is witnessed by a linear polynomial (see proposition \ref{linearreduction}) does not extends to sets that are infinite and coinfinite.

Indeed, using the notation of the proof of proposition \ref{longzchain}, let
\[
A=\bigcup_{m\in \Z } \mathcal C_2^{2m}\cap k,\qquad B=\bigcup_{m\in \Z } \mathcal C_2^{2m+1}\cap k.
\]
Letting $P(X)=X^2$, it follows that $A=P^{-1}(B),B=P^{-1}(A)$, whence $A \eqpol B$.
However there is no linear polynomial $Q$ such that $A=Q^{-1}(B)$, since otherwise, considering the polynomial function $ \hat Q : \C \to \C $ extending $Q$, by the density of $k$ in $ \C $ it would follow that $\bigcup_{m\in \Z } \mathcal C_2^{2m}= \hat Q^{-1} \left (
\bigcup_{m\in \Z } \mathcal C_2^{2m+1}
\right ) $, which is impossible.
\end{remark}

\section{Further questions}
We have restricted our study to the case of an algebraically closed field $k$ of characteristic $0$.
While it is apparent that the hypothesis of algebraic closure has been widely used, we point out that the fact that $k$ has characteristic $0$ is needed in the the proof of proposition \ref{linearreduction} and of lemma \ref{30112022}, and therefore plays a central role both for the combinatorics of $ \pol $ and for the geometric structure of the polynomial classes.

These observations raise the following questions (see also proposition \ref{orderedfields}).

\begin{question}
What can be said about $ \pol $ when $k$ is not algebraically closed?
\end{question}

\begin{question}
What can be said about $ \pol $ when $k$ has positive characteristic?
\end{question}

The polynomial functions $k\to k$, which we have chosen as the class of reductions inducing the preorder $ \pol $ we have studied in this paper, are the morphisms of the affine variety $k$.
Similarly, the polynomial functions $V\to V$ are the morphisms of $V$, for any affine variety $V$.
It is then natural to try to extend our study, starting perhaps with the simplest case, namely $k^n$.

\begin{question} \label{19112022}
Study the relation of polynomial reducibility on $k^n$ and on other affine varieties.
To what extent the structure of polynomial reducibility is sensible to the geometric properties of the variety?
\end{question}

An answer to the last part of question \ref{19112022} should be compared with the results of \cite{cammas}, where it is shown that most of the geometric properties of the variety are forgotten by the relation of continuous reducibility $\le_W$, whose structure turns out to depend only on few features of the variety, like the cardinality, the dimension, and the decomposition into irreducible components.

\section{Appendix: a variation on Vandermonde} \label{310720221554}
In this section we establish the claim in the proof of lemma \ref{020820221452}.

Given a matrix $M$, denote $R_1^M,R_2^M,\ldots $ the rows of $M$ and by $C_1^M,C_2^M,\ldots $ the columns of $M$.
Let also $M[i,j]$ be the entry of $M$ on the $i$-th row and $j$-th column.

\begin{definition} \label{040820221656}
Let $\gamma ,h$ be natural numbers, with $h\ge 1$.
Fix a sequence $ \vec s =(s_1,\ldots ,s_h)$ of natural numbers and a sequence $ \vec a=(a_1,\ldots ,a _h)$ of distinct variables.
Define the sequence $ \vec r =(r_1,\ldots ,r_h)$ by $ \begin{cases}
r_1=1 \\
r_{\ell +1}=r_{\ell }+s_{\ell }+1 \text{ for } 1\le\ell <h
\end{cases} $.
Assume that $R=\sum_{\ell =1}^hs_{\ell }+h\le\gamma +1$.

We call \emph{enriched Vandermonde matrix}, associated to the triple $(\gamma +1, \vec s , \vec a )$, the matrix $M$ with $R$ rows and $\gamma +1$ columns such that
\[
M[i,j]= \frac{d^s}{da_{\ell }^s} a_{\ell }^{j-1},
\]
where $\ell ,s$ are unique with
\begin{equation} \label{040820222128}
i=r_{\ell }+s,\qquad 0\le s\le s_{\ell }.
\end{equation}
\end{definition}

In other words, an enriched Vandermonde matrix is a matrix, with a number of rows not exceeding the number of columns, consisting of blocks of rows of the form \eqref{04082022}.
Notice that the matrix $M$ of the claim in lemma \ref{020820221452} is an enriched Vandermonde matrix (though its blocks are one row shorter, as in its definition the last line of each block has been erased).

\begin{theorem}
Let $M$ be an enriched Vandermonde matrix, with $R$ rows and $\gamma +1$ columns.
Then $rkM=R$.
\end{theorem}

\begin{proof}
Use the notation of definition \ref{040820221656}.
Consider the homogeneous linear system associated to matrix $M$, in the unknowns $(c_0,c_1,\ldots ,c_{\gamma })$.
Then $(c_0,c_1,\ldots ,c_{\gamma })$ is a solution to the system if and only if the polynomial $P(X)=\sum_{\ell =0}^{\gamma }c_{\ell }X^{\ell }$ has $a_j$ as root of multiplicity at least $s_j+1$, since the equations in each block mean $P(a_j)=P'(a_j)=\ldots =P^{(s_j)}(a_j)=0$.
Let $V$ be the subvariety of $k^{\gamma +1}$ of the solutions of the system.
Then $V$ is a vector subspace of $k^{\gamma +1}$ of dimension $\gamma +1-rkM$, so it is enough to prove that $\dim (V)=\gamma +1-R$.

Let $\varphi :k^{\gamma +1}\to k^{\gamma +1}$ be the morphism defined by letting
\begin{multline*}
\varphi (c,\alpha_1,\ldots ,\alpha_{\gamma })= \\
=((-1)^{\gamma }cE_{\gamma }(\alpha_1,\ldots ,\alpha_{\gamma }),(-1)^{\gamma -1}cE_{\gamma -1}(\alpha_1,\ldots ,\alpha_{\gamma }),\ldots ,-cE_{1}(\alpha_1,\ldots ,\alpha_{\gamma }),c),
\end{multline*}
where $E_1,\ldots ,E_{\gamma }$ are the elementary symmetric functions, so that $\varphi (c,\alpha_1,\ldots ,\alpha_{\gamma })$ is the sequence of coefficients of the polynomial $c(X-\alpha_1)(X-\alpha_2)\cdot\ldots\cdot (X-\alpha_{\gamma })$.

Let $ \mathcal F =\{ F_1,\ldots ,F_h\} $ be a collection of disjoint subsets of $\{ 1,\ldots ,\gamma\} $ with $\card (F_j)=s_j+1$ for every $j\in\{ 1,\ldots ,h\} $, and let $W$ be the subvariety of $k^{\gamma +1}$, in the coordinates $c,\alpha_1,\ldots ,\alpha_{\gamma }$, obtained by setting $\alpha_r=a_j$ for every $r\in F_j,j\in\{ 1,\ldots ,h\} $.
Therefore $\dim (W)=\gamma+1-R$.
Now note that the fibers of $\varphi $ are finite and that $\varphi (W)=V$, so the result follows from \cite[corollary IV.3.8(2)]{perrin2008}.
\end{proof}

\end{document}